\pdfoutput=1                      

\documentclass[10pt,a4paper]{article}

\usepackage[accsupp]{axessibility}    

\usepackage{latexsym,amsfonts,amsmath,amsthm,amssymb, graphicx, mathrsfs, fancyhdr,
amscd,
fontenc,verbatim,stix,amsfonts,amsthm,indentfirst}

\numberwithin{equation}{section} 
\usepackage[all]{xy}
\usepackage{braket}

\usepackage{color,xcolor}

\usepackage[draft=false,setpagesize=false,pdfstartview=FitH,
colorlinks=true,citecolor=blue,pagebackref=true]{hyperref}

\theoremstyle{definition}
\newtheorem{definition}{Definition}[section]
\newtheorem{example}[definition]{Example}

\theoremstyle{plain}
\newtheorem{proposition}[definition]{Proposition}

\newtheorem{theorem}[definition]{Theorem}
\newtheorem{corollary}[definition]{Corollary}
\newtheorem{note}[definition]{Note}

\newtheorem*{conjecture*}{Conjecture}

\theoremstyle{definition}
\newtheorem{remark}[definition]{Remark}
\newtheorem{notation}[definition]{Notation}

\newcommand{\diver}{\mathrm{div}}

\DeclareMathOperator{\ch}{\mathrm{ch}}

\newcommand{\eps}{\varepsilon}

\newcommand{\di}{\mathrm{d}}

\newcommand{\R}{\mathbb R}

\newcommand{\lip}{\mathrm{Lip}}
\newcommand{\loc}{\mathrm{loc}}
\newcommand{\disp}{\displaystyle}

\newcommand{\capac}{\mathrm{cap}}

\newcommand{\Ric}{\mathrm{Ric}}
\newcommand{\Sec}{\mathrm{Sec}}
\newcommand{\BRic}{\mathrm{BRic}}

\newcommand{\pa}[1]{{\left(#1\right)}}                  

\makeatletter      
\newcommand*\owedge{\mathpalette\@owedge\relax}
\newcommand*\@owedge[1]{
	\mathbin{
		\ooalign{
			$#1\m@th\bigcirc$\cr
			\hidewidth$#1\m@th\wedge$\hidewidth\cr
		}
	}
}
\makeatother

\newcommand{\sq}[1]{{\left[#1\right]}}

\newcommand{\So}{\mathscr{S}}

\allowdisplaybreaks[1]

\begin{document}

\author{Giovanni Catino, Luciano Mari, Paolo Mastrolia, Alberto Roncoroni}
\title{\textbf{Criticality, splitting theorems under spectral Ricci bounds and the topology of stable minimal hypersurfaces}}

\newcommand{\Addresses}{{
  \bigskip
  \footnotesize

  G.~Catino, \textsc{Dipartimento di Matematica, Politecnico di Milano, Piazza Leonardo da Vinci 32, 20133, Milano, Italy.}\par\nopagebreak
  \textit{E-mail address}, G.~Catino: \texttt{giovanni.catino@polimi.it}

  \bigskip

  L.~Mari, \textsc{Dipartimento di Matematica, Universit\`a degli Studi di Milano,
Via Cesare Saldini 50, 20133 Milano, Italy}\par\nopagebreak
  \textit{E-mail address}, L.~Mari: \texttt{luciano.mari@unimi.it}

  \bigskip

  P.~Mastrolia, \textsc{Dipartimento di Matematica, Universit\`a degli Studi di Milano,
Via Cesare Saldini 50, 20133 Milano, Italy}\par\nopagebreak
  \textit{E-mail address}, P.~Mastrolia: \texttt{paolo.mastrolia@unimi.it}

 \bigskip

  A.~Roncoroni, \textsc{Dipartimento di Matematica, Politecnico di Milano, Piazza Leonardo da Vinci 32, 20133, Milano, Italy.}\par\nopagebreak
  \textit{E-mail address}, A.~Roncoroni: \texttt{alberto.roncoroni@polimi.it}

}}

\date{}

\maketitle

\begin{center}
	\textit{To the memory of Francesco ``Franco'' Mercuri and Celso Viana}
\end{center}

\normalsize

\

\begin{abstract}
	In this paper we prove general criticality criteria for operators $\Delta + V$ on manifolds with more than one end, where $V$ bounds the Ricci curvature, and a related spectral splitting theorem extending Cheeger \& Gromoll's one. Our results give new insight on Li \& Wang's theory of manifolds with a weighted Poincar\'e inequality. We apply them to study stable and $\delta$-stable minimal hypersurfaces in manifolds with non-negative bi-Ricci or sectional curvature, in ambient dimension up to $5$ and $6$, respectively. In the special case where the ambient space is $\R^4$, we prove that a $1/3$-stable minimal hypersurface must either have one end or be a catenoid, and that proper, $\delta$-stable minimal hypersurfaces with $\delta > 1/3$ must be hyperplanes.
\end{abstract}

\

\

\noindent
\textbf{MSC2020}: 53C24, 53C21, 53C42.

\noindent
\textbf{Keywords:} Criticality, splitting, stable minimal hypersurfaces, spectral Ricci bounds, catenoid.



\section{Introduction}

A classical theme in Riemannian geometry is the study of the structure of manifolds whose Ricci curvature is bounded from below. To this aim, if $\Ric \ge 0$ a major result is given by Cheeger \& Gromoll's splitting theorem \cite{chgr}. However, when the Ricci curvature is negative somewhere, to control the topology of $M$ (for instance, the number of its ends) one needs to complement the Ricci lower bound with some further assumptions saying, roughly speaking, that the negative part of the Ricci curvature is not too large compared to suitable data on the global behaviour of $M$. Among the conditions that were proposed in the literature, a \emph{spectral Ricci lower bound} is  natural in view of applications, see \cite{prs_book}. By a spectral Ricci lower bound we mean the existence of $V \in C^{0,\alpha}_\loc(M)$ and $\beta \in \R^+$ so that
\begin{equation}\label{assu_basic_L}
	\Ric \ge - \beta V g, \qquad L \doteq -\Delta - V \ge 0,
\end{equation}
where $L \ge 0$ means that the operator is non-negative in the spectral sense, i.e. the associated quadratic form 
\[
\int_M |\nabla \phi|^2 - V \phi^2 
\]
is non-negative for each $\phi \in C^1_c(M)$. Note that assuming \eqref{assu_basic_L} for some $V$ is equivalent to saying that 
\[
-\beta \Delta + \lambda_{\Ric} \ge 0, 
\]
where $\lambda_\Ric : M \to \R$ is the lowest eigenvalue of the Ricci tensor\footnote{It is known that $\lambda_\Ric$ is a locally Lipschitz function (cf. \cite{hw}), so assuming $V \in C^{0,\alpha}_\loc(M)$ is not restrictive.}. Relevant examples of such manifolds include stable minimal hypersurfaces $M^n \to N^{n+1}$ into non-negatively curved ambient spaces, for which one has
\[
\Ric \ge - \frac{n-1}{n}|A|^2 g
\]
by Gauss equation, while the stability condition implies $-\Delta -|A|^2 \ge 0$. These will be the focus of our geometric applications. Indeed, the study of manifolds with a spectral Ricci bound gained new momentum in connection to the recent almost complete solution to the Stable Bernstein Problem, that is, the question whether complete, two-sided stable minimal hypersurfaces $M^n \to \R^{n+1}$ must be hyperplanes if $n \le 6$. The positive answers obtained by Chodosh \& Li \cite{cl_1,cl_2} ($n=3$), Catino, Mastrolia \& Roncoroni \cite{cmr} ($n=3$), Chodosh, Li, Minter \& Stryker \cite{clms} ($n=4$) and Mazet \cite{mazet} ($n=5$), and  the results by Bellettini \cite{bellettini} for $n=6$, provided new powerful tools with promising applications to treat more general ambient spaces. \par

Versions of the splitting theorem under spectral Ricci lower bounds were first studied in the seminal works of Li \& Wang \cite{liwang_positive_1,liwang_positive_2} (in the case of constant $V$) and \cite{liwang_ens}. For suitable ranges of $\beta$, the authors characterized warped product structures among manifolds satisfying \eqref{assu_basic_L} and having at least two ends (in some cases, two non-parabolic ends). Their approach differs from that of Cheeger and Gromoll: in particular, the use of Busemann functions is replaced by harmonic functions, relying on the theory developed by Li and Tam \cite{litam} (see \cite{prs_book} for a detailed exposition). In this analysis, the need to restrict to
\[
\beta \le \frac{n-1}{n-2}
\]
arises naturally, and this bound is in fact sharp, as demonstrated by the recent gluing construction of Antonelli \& Xu \cite{anto_xu_gluing}. Moreover, obtaining rigidity in the range
\[
\beta \in \left(1, \frac{n-1}{n-2}\right]
\]
is particularly delicate. When $V$ is not constant, Li and Wang were forced to impose additional technical assumptions which, although quite general and later substantially refined by Cheng \& Zhou \cite{cheng_zhou_1}, still limit the applicability of the main structural results to stable minimal hypersurfaces.


In this paper, we study manifolds satisfying \eqref{assu_basic_L} and obtain new results complementing Li \& Wang's ones. Hereafter, 
\[
\text{\emph{All manifolds $M$ will be assumed connected.}}
\]
Among others, we mention the following spectral splitting theorem which generalizes Cheeger \& Gromoll's \cite{chgr} and Cheng's \cite{xucheng}: 

\begin{theorem}\label{chegrointro}
	Let $(M,g)$ be a complete Riemannian manifold of dimension $n \ge 3$, and assume that there exists $V \in C^{0,\alpha}_\loc(M)$ such that
	\[
	\Ric \ge - \beta V g \qquad \text{on } \, M, \qquad -\Delta - V \ge 0,
	\]
	where either
	\[
	\begin{array}{ll}
		\disp \quad \quad 0<\beta < \frac{4}{n-1}, & \quad \text{or }  \\[0.5cm]
		\disp \frac{4}{n-1} \le \beta < \frac{n-1}{n-2}  & \quad \text{and $V_+$ is  compactly supported.}
	\end{array}
	\]
	Then, either
	\begin{itemize}
		\item[(i)] $M$ has only one end, or
		\item[(ii)] $V \equiv 0$ and $M = \R \times P$ with the product metric, for some compact $P$ with $\Ric_P \ge 0$.	
	\end{itemize}
\end{theorem}

%

\begin{remark}
	The case $\beta < \frac{4}{n-1}$ was also obtained, independently and at the same time, by Antonelli, Pozzetta \& Xu \cite{apx} with a different technique, see below.
\end{remark}

Before explaining the results and their proofs in more detail, we describe our applications.

\subsection*{Applications to minimal hypersurfaces}

We consider two-sided, complete minimal hypersurfaces $M^n \to (N^{n+1},\bar g)$ immersed into an ambient space $N$ with non-negative sectional curvature $\overline{\Sec}$ or, more generally, non-negative bi-Ricci curvature $\overline{\BRic}$. Recall that, given an orthonormal set $\{X,Y\} \subset T_pN$, the bi-Ricci curvature is defined as
\[
\overline{\BRic}(X,Y) = \overline{\Ric}(X,X) + \overline{\Ric}(Y,Y) - \overline{\Sec}(X,Y).
\]
In particular, if $n+1=3$ then $\overline{\BRic} = \bar R/2$, with $\bar R$ the scalar curvature of $N$. 

We first suppose that $M$ is stable, i.e. that the Jacobi operator 
\[
-\Delta - \left( |A|^2 + \overline{\Ric}(\nu,\nu)\right)
\]
is non-negative. Here, $\nu$ is a global unit normal vector field on $M$ and $A$ is the second fundamental form in direction $\nu$. 

Compared to the extensive literature on minimal surfaces in non-negatively curved three-manifolds, relatively few results are available in higher dimensions. In particular, aside from some special cases (notably the one in \cite{cls}), the Stable Bernstein Theorem -- namely, the statement that complete, two-sided stable minimal hypersurfaces must be totally geodesic -- remains open in ambient manifolds of dimension $n +1 \ge 4$ with $\overline{\Sec} \ge 0$ other than $\R^{n+1}$. Structural results for stable minimal hypersurfaces, such as those established below, may therefore be viewed as a first step in this direction.
	
In the following list, we collect the results most closely related to the present paper. We assume $n \ge 3$ and, for simplicity, that $N$ is complete.

%

\begin{itemize}
	\item Li \& Wang \cite[Thm. 0.1]{liwang_minimal}: if $\overline{\Sec} \ge 0$ and $M$ is properly immersed, then either $M$ has only one end or $M$ is totally geodesic and $M = \R \times P$ with the product metric, for some compact $P$ with non-negative sectional curvature.
	\item Cheng \cite[Thm. 3.2]{xucheng}: if $\overline{\Sec} \ge 0$ and $\overline{\Ric} > 0$, then $M$ is non-parabolic and it has only one end.
	\item Chodosh, Li \& Stryker \cite[Thm. 1.3 and Remark 1.4]{cls}: if $n=3$, $0 \le \overline{\Sec} \le \kappa$ for some $\kappa \in \R^+$ and the scalar curvature of $N$ satisfies $\bar R \ge 1$, then $M$ is totally geodesic and $\overline{\Ric}(\nu,\nu) \equiv 0$.
	\item Tanno \cite[Thm. B]{tanno}: if $\overline{\BRic} \ge 0$ and $M$ is orientable, then the space of $L^2$ harmonic $1$-forms on $M$ is trivial. In particular, Li-Tam's theory in \cite{litam} implies that $M$ has only one non-parabolic end, and by \cite{dodziuk}, every cycle of codimension $1$ must disconnect $M$. The result was obtained in \cite{miya} by assuming $\overline{\Sec} \ge 0$, see also \cite{palmer,vieira}. Related sharp vanishing theorems on spinnable minimal hypersurfaces were recently obtained by Bei \& Pipoli \cite{bei_pipoli}.	
	\item Shen \& Zhu \cite[Thm. 1 and Rem. 2]{shen_zhu_MathAnn}: if $n \in \{3,4\}$, $\overline{\BRic} \ge 0$ and $M$ is topologically $\R \times P$, where $P$ is compact and admits a metric of non-positive sectional curvature, then $M$ is totally geodesic. Moreover, $M$ is flat if $\overline{\Sec} \ge 0$.
	\item Cheng \cite[Thm. 2.1]{xucheng}: if $n \in \{3,4\}$ and  $\overline{\BRic} > 0$, then $M$ must have only one end\footnote{Cheng assumed $M$ and $N$ orientable, but her proof only uses that $M$ is two-sided in $N.$}.
\end{itemize}


We are now ready to state our main applications. We first consider ambient manifolds with non-negative bi-Ricci curvature, and obtain the following splitting theorem that improves on \cite{xucheng,shen_zhu_MathAnn}.

%
%

\begin{theorem}\label{teo_minimal_main_stable}
	Let $M^n \to (N^{n+1},\bar g)$ be a complete, non-compact, two-sided stable minimal hypersurface of dimension $n \in \{3,4\}$ in a manifold satisfying
	\[
	\overline{\BRic} \ge 0.
	\]
	Then, one of the following mutually exclusive cases occurs:
	\begin{itemize}
		\item[(i)] $M$ has only one end;
		\item[(ii)] $M$ is totally geodesic and $M =  \R \times P$ with the product metric $\di t^2 + g_P$, for some compact $(n-1)$-manifold $(P,g_P)$ with non-negative Ricci curvature. Moreover,
		\[
		\overline{\Ric}(\nu,\nu) \equiv 0, \qquad  \overline{\BRic}(\partial_t,\nu) \equiv 0.
		\]
	\end{itemize}
	Furthermore, if $\overline{\Ric} \ge 0$ then in (i) either $M$ is non-parabolic, or $M$ is parabolic, totally geodesic and $\overline{\Ric}(\nu,\nu) \equiv 0$.
\end{theorem}


\begin{remark}
	Observe that $\overline{\BRic}\ge 0$ does not imply $\overline{\Ric} \ge 0$, so the conclusion $\overline{\Ric}(\nu,\nu) \equiv 0$ in $(ii)$ is not obvious. In dimension $n=2$, $\overline{\BRic}\ge 0$ becomes $\bar R \ge 0$, hence Theorem \ref{teo_minimal_main_stable} can be seen as a higher dimensional analogue of the following well-known result by Fischer-Colbrie \& Schoen, \cite[Theorem 3]{fischer_schoen}: a stable minimal surface $M^2 \to N^3$ with more than one end in an ambient space with $\bar R \ge 0$ must be conformally equivalent to a cylinder, flat and totally geodesic provided it has finite total curvature, i.e.
	\[
	\int_M |K| < \infty.
	\]
	In the same paper, the authors conjectured that the request of finite total curvature should be removable. This is the case, as shown by \cite{miya,schoen_yau,berard_castillon}. As a matter of fact, the argument in  Theorem \ref{teo_minimal_main_stable} works \emph{verbatim} also in dimension $n=2$ (with $P = \mathbb{S}^1$), thus providing a further different proof of Fischer-Colbrie and Schoen's claim.
\end{remark}

We next consider ambient spaces with non-negative sectional curvature. Our goal is to improve on Li \& Wang's remarkable splitting theorem \cite[Thm. 0.1]{liwang_minimal} quoted above by removing the condition that the immersion $M \to N$ is proper. As a matter of fact, Li \& Wang's result is the consequence of a more general statement where properness is not required a priori, see \cite[Thm. 4.1]{liwang_minimal}. However, in counting the number of ends the major issue therein is the possible presence of parabolic ends which are extrinsically bounded. In Theorem \ref{teo_minimal_main_stable} we are able to overcome the problem in dimension $n \le 4$. In the next theorem, we include one more dimension:

\begin{theorem}\label{teo_minimal_main_stable_sec}
	Let $M^n \to (N^{n+1},\bar g)$ be a complete, non-compact, two-sided stable minimal hypersurface of dimension $n \in \{3,4,5\}$ in a manifold satisfying
	\[
	\overline{\Sec} \ge 0.
	\]
	Then, one of the following cases occurs:
	\begin{itemize}
		\item[(i)] $M$ has only one end. Moreover, either $M$ is non-parabolic, or $M$ is parabolic, totally geodesic and $\overline{\Ric}(\nu,\nu) \equiv 0$;
		\item[(ii)] $M$ is totally geodesic, $\overline{\Ric}(\nu,\nu) \equiv 0$ and $M =  \R \times P$ with the product metric, for some compact $(n-1)$-manifold $P$ with non-negative sectional curvature.
	\end{itemize}	
\end{theorem}

\begin{remark}
	Differently from \cite{liwang_minimal}, Theorems \ref{teo_minimal_main_stable} and \ref{teo_minimal_main_stable_sec} do not need $N$ to be complete.
\end{remark}


\begin{remark}
	Theorem 0.1 in \cite{liwang_minimal} has a counterpart for properly immersed minimal hypersurfaces $M \to N$ with finite index, which are shown to have finitely many ends, see \cite[Thm. 4.1]{liwang_minimal}. We think it would be very interesting to prove an analogous result in dimension $n \le 5$ by removing the properness assumption.
\end{remark}

As a consequence of arguments by Shen \& Sormani \cite{shen_sormani} and Carron \& Pedon \cite{car_ped}, we also obtain some information on the (co)-homology of $M$. Here, $H^p_c(M)$ is De Rham's $p$-th compactly supported cohomology group of $M$.
\begin{corollary}\label{cor_topol_minimal}
In the assumptions of either Theorem \ref{teo_minimal_main_stable} or Theorem \ref{teo_minimal_main_stable_sec}, if $M$ has only one end, then one of the following cases occurs:  
\begin{itemize}
			\item[-] $H^1_c(M) = 0$ and, if $M$ is orientable, $H_{n-1}(M, \mathbb{Z}) = 0$;			
			\item[-] $M$ is totally geodesic and it is the determinant line bundle of a compact manifold $(P,g_P)$ with locally the product metric $\di t^2 + g_P$, where $t$ is the arclength of the fibers. Moreover, 
			\[
			\overline{\Ric}(\nu,\nu) \equiv 0, \qquad \overline{\BRic}(\partial_t,\nu) \equiv 0. 
			\]
			In particular, $H_{n-1}(M, \mathbb{Z}) = 0$ if $M$ is orientable, and $H_{n-1}(M, \mathbb{Z}) = \mathbb{Z}$ if $M$ is nonorientable.
		\end{itemize}
Moreover, in the assumptions of Theorem \ref{teo_minimal_main_stable_sec}, if $n=3$ then $H_2(M, \mathbb{Z})=0$ holds in the first case independently of the orientability of $M$. 
\end{corollary}
Lastly, we focus on minimal hypersurfaces which are $\delta$-stable, i.e. whose modified Jacobi operator
\[
-\Delta - \delta \left( |A|^2 + \overline{\Ric}(\nu,\nu)\right)
\]
is non-negative. Standard stability corresponds to $\delta =1$, and one easily checks that $\delta$-stability implies $\delta'$-stability if $\delta'<\delta$. The notion originated in the context of Colding--Minicozzi's theory for minimal surfaces \cite{cm_1, cm_2}, and $\delta$-stable minimal surfaces are examples of surfaces whose Gaussian curvature $K$ satisfies  
\[
-\Delta + \beta K - W \ge 0
\]
for some $\beta \ge 0$ and $W \in C(M)$. Beginning with \cite{fischer_schoen}, this class has been extensively investigated in the literature, see \cite{berard_castillon} for a detailed account of their structure. We here consider dimension $n \ge 3$ and the range
\[
\delta \in \left[\frac{n-2}{n},1\right],
\]
corresponding to a spectral Ricci bound \eqref{assu_basic_L} with $\beta \in \left[\frac{n-1}{n}, \frac{n-1}{n-2}\right]$. Our choice to include $\delta$-stable hypersurfaces in our investigation is also motivated by the following beautiful gap result:

\begin{theorem}[\cite{anderson,shen_zhu,fuli,pv}]\label{teo_pv}
	Let $M^n \to \R^{n+1}$ be a complete, two-sided minimal hypersurface with finite total curvature, i.e. $|A| \in L^n(M)$. Then:
	\begin{itemize}
		\item $M$ is $\delta$-stable with $\delta > \frac{n-2}{n}$ if and only if $M$ is a hyperplane;
		\item $M$ is $\frac{n-2}{n}$-stable if and only if $M$ is either a hyperplane or a catenoid.
	\end{itemize}
\end{theorem}
\begin{remark}
The theorem was proved by Anderson \cite{anderson} and Shen \& Zhu \cite{shen_zhu} for $\delta =1$. Its extension to the present form is due to Fu \& Li \cite{fuli}, based on works \cite{anderson,cheng_zhou_1}. A different, self-contained proof was found by Pigola \& Veronelli \cite{pv}, which we also recommend for a thorough account of the literature. The $\frac{n-2}{n}$-stability of the higher dimensional catenoid and the finiteness of its total curvature were proved by Tam \& Zhou in \cite{tam_zhou}.
\end{remark}

It is tempting to ask whether Theorem \ref{teo_pv} holds in dimension $n \le 6$ (or in a smaller range) without assuming that $M$ has finite total curvature, in other words, if the Stable Bernstein Theorem can be improved to the conclusions of Theorem \ref{teo_pv}. In \cite[Theorem E.1]{cl_1}, the authors obtained a first step in this direction in dimension $n=3$ by proving that $2/3$-stability implies flatness, see also \cite[Proposition E.2]{cl_1}. Recently, the result was improved by Hong, Li \& Wang to $\delta$-stability with $\delta > 3/8$, see \cite[Theorem 1.7]{hlw}. The paper \cite{hlw} also contains various interesting results in the full range $\delta > \frac{n-2}{n}$.
%
As a consequence of \cite{hlw,cl_1} and our structure theorems, we obtain a $\delta$-stable Bernstein Theorem in $\R^4$ under the optimal bound on $\delta$:
%
\begin{theorem}\label{teo_bernstein}
	Let $M^3 \to \R^4$ be a $2$-sided, properly immersed minimal hypersurface which is $\delta$-stable for some $\delta > 1/3$. Then, $M$ is a hyperplane.
\end{theorem}
A bound on the number of ends of $\delta$-stable hypersurfaces is a main missing information we provide to get Theorem \ref{teo_bernstein}. In this respect, observe that $\delta$-stability for $\delta$ close to $\frac{n-2}{n}$ corresponds to a spectral Ricci bound \eqref{assu_basic_L} with $\beta$ close to $\frac{n-1}{n-2}$, a case for which structure theorems are harder to achieve, see \cite{liwang_ens,cheng_zhou_1}. For example, it is known by \cite{cao_shen_zhu} that stable minimal hypersurfaces in $\R^{n+1}$ must have only one end, a fact that plays a crucial role in the proofs in \cite{cl_2,clms,mazet}.
Although the argument in \cite{cao_shen_zhu} easily adapts to $\delta$-stability for any
\[
\delta \ge \frac{n-1}{n},
\]
justifying the $2/3$-stability request in  \cite{cl_1}, it fails for $\delta < \frac{n-1}{n}$. To the best of our knowledge, in full generality no bound on the number of ends of $M$ was established so far for $\delta$ close to $\frac{n-2}{n}$, and neither the assertion that $M$ has finitely many ends. We are only aware of the following result due to Cheng \& Zhou, \cite[Theorem 1.1]{cheng_zhou_1}:

\begin{theorem}[\cite{cheng_zhou_1}]\label{teo_chengzhou}
	Let $M^n \to \R^{n+1}$ be a $\frac{n-2}{n}$-stable, two-sided complete minimal hypersurface of dimension $n \ge 3$. If
	\begin{equation}\label{eq_growth_A}
		\begin{array}{ll}
			\disp \lim_{R \to \infty} \left(\sup_{B_R(o)} |A|\right) / \log R = 0 & \quad \text{if } \, n =3 \\[0.5cm]
			\text{ or} \\
			\disp \lim_{R \to \infty} \left(\sup_{B_R(o)} |A|\right) / R^{\frac{n-3}{2}} = 0 & \quad \text{if } \, n \ge 4,
		\end{array}
	\end{equation}
	where $B_R(o) \subset M$ is the intrinsic ball centered at $o$ with radius $R$, then either $M$ has only one end or it is a catenoid.
\end{theorem}

Observe that \eqref{eq_growth_A} holds, for instance, if $|A|$ is bounded on $M$. Cheng and Zhou kindly suggested to the first and second author the question whether \eqref{eq_growth_A} is removable, which indeed was the starting point of the present investigation. We are able to positively answer in dimension $n=3$:

\begin{theorem}\label{cor_chengzhou}
	Let $M^3 \to \R^{4}$ be a complete, two-sided, $1/3$-stable minimal hypersurface. Then, either $M$ has only one end or it is a catenoid.
\end{theorem}

Theorem \ref{cor_chengzhou} is a corollary of a more general result, Theorem \ref{teo_minimal_main} below, where we study the topology of $1/3$-stable minimal hypersurfaces $M^3 \to N^4$ into a $4$-manifold with non-negative sectional curvature. The peculiarity of dimension $n=3$ will be clear in what follows, see in particular Theorem \ref{teo_main_critical}.

\subsection*{Structure theorems and novelties of our approach}

Besides Li \& Wang's \cite{liwang_positive_1,liwang_positive_2,liwang_ens},  
manifolds satisfying 
\[
\Ric \ge - \beta V g, \qquad L \doteq -\Delta - V \ge 0
\]
were studied by Cheng \cite{xucheng}, Cheng \& Zhou \cite{cheng_zhou_1} and Pigola, Rigoli \& Setti \cite{prs_JFA,prs_book} (see also B\'erard \cite{berard}). More recently, we mention the works of Bour \& Carron \cite{boucar}, Carron \cite{car}, Carron \& Rose \cite{carrose} and Antonelli \& Xu \cite{antonelli_xu,anto_xu_gluing}. The investigation highlighted three ranges for $\beta$:

\begin{itemize}
	\item[(i)] $\beta \le \frac{n-1}{n-2}$: in this range, under some further assumptions on $M$ one is able to control the number of non-parabolic ends, see \cite[Theorem A and Corollary 4.2]{liwang_ens}.
	\item[(ii)] $\beta \le \frac{4}{n-1}$: this range also allows to control the number of parabolic ends of $M$, see \cite[Theorems B,C]{liwang_ens} and \cite{xucheng}.
	\item[(iii)] $\beta < \frac{1}{n-2}$: in this range, $M$ looks close to a manifold with $\Ric \ge 0$, as it will become apparent in Theorem \ref{teo_main_critical_3} below. Analogies become tighter if one further assumes the \emph{gaugeability} of $L$, i.e. the existence of a positive solution to $\Delta u + Vu \le 0$ bounded below and above by positive constants, see \cite{car,carrose,irv}.	
\end{itemize}

\begin{remark}\label{rem_n3}
	Note that $\frac{n-1}{n-2} \ge \frac{4}{n-1}$, with equality if and only if $n=3$. This accounts for most of the extra information we obtain in dimension $3$.
\end{remark}

%

In this paper, we consider each of the ranges  (i),(ii),(iii). Differently from \cite{liwang_ens} and from \cite{cheng_zhou_1,car,carrose,antonelli_xu}, our structure theorems
rely on a new point of view on the ``conformal method" pioneered by Schoen \& Yau \cite{schoen_yau_cmp}, Fischer-Colbrie \cite{fischercolbrie}, Shen \& Ye \cite{shen_ye_min} and Shen \& Zhu \cite{shen_zhu_MathAnn}. The method exploits a positive solution to 
\[
\Delta u + Vu \le 0, 
\]
whose existence is granted by the second in \eqref{assu_basic_L}, to conformally change the metric via $\bar g = u^{2\beta} g$. This allows to ``push-up" the curvature contribution in the second variation formulas for $\bar g$-geodesics. However, so far the idea was mostly used to prove compactness results in the spirit of Bonnet \& Myers' one, see \cite{schoen_yau,shen_ye_min,shen_ye,xucheng_H,enr,mmrs,cr}. 
Exceptions are \cite{fischercolbrie,xucheng,cmr} and \cite{cmm}, where it was employed to control the geometry of complete minimal hypersurfaces and to prove the rigidity of critical metrics for a quadratic curvature functional, respectively. 
Still, the splitting problem was not addressed there, and requires to rework the method from the onset.

A major issue appearing, for instance, in the proof of Theorem \ref{chegrointro} is to extend the information obtained on a single $\bar g$-geodesic to the whole manifold. To this aim,
an original contribution of the present paper is to link the conformal method to the {\em criticality theory} for operators $-\Delta-V \ge 0$ developed by Murata \cite{murata} and Pinchover \& Tintarev \cite{pinchover, pinchover2, pincho_tinta,pincho_tinta_p} (see also  Zhao \cite{zhao}). Roughly speaking, the criticality theory extends the standard dichotomy between parabolic and non-parabolic manifolds to the case of possibly nonzero potentials $V$. 
The conclusions of our main Theorems \ref{teo_main_critical} and \ref{teo_main_critical_2} (i.e. Theorem \ref{chegrointro}) can be summarized as follows: either $M$ has only one end, or $-\Delta - V$ is critical and the first inequality in \eqref{assu_basic_L} is saturated at every point by at least one vector. More precisely:

\begin{itemize}
	\item in Theorem \ref{teo_main_critical}, we obtain such a conclusion in the range (i). However, we need a further constraint in dimension $n \ge 4$ (see Example \ref{ex_1_liwang}), in accordance to statement (i) above and Remark \ref{rem_n3}: under condition $\beta \le \frac{n-1}{n-2}$ one hardly detects parabolic ends if $n \ge 4$. The constraint is not necessary if $V_+$ is compactly supported, i.e. if $\Ric \ge 0$ outside of a compact set, see Theorem \ref{teo_main_critical_cs};
	\item in Theorem \ref{teo_main_critical_2}, Theorem \ref{chegrointro} is proved. A key point in the argument is a convexity lemma in criticality theory, see Theorem \ref{teo_subcrit}. 
\end{itemize}
\noindent In the range
\[
\beta < \frac{1}{n-2},
\]
we prove in Theorem \ref{teo_main_critical_3} that for \emph{any} choice of a positive (smooth) solution to $\Delta u + Vu \le 0$, the metric
\[
\bar g = u^{2\beta} g
\]
is a \emph{complete} metric with non-negative modified Bakry-\'Emery Ricci tensor
\[
\overline{\Ric}_f^N \doteq \overline{\Ric} + \bar{\nabla}^2f -\frac{1}{N-n} \di f\otimes \di f,
\]
for $f = (n-2)\beta \log u$ and a suitable $N = N(n,\beta) > n$. The result, which we obtain by adapting \cite{cmr}, establishes a higher-dimensional analogue of Fischer-Colbrie's \cite[Theorem 1]{fischercolbrie} and enables to exploit the extensive literature on manifolds with 
\[
\overline{\Ric}_f^N \ge 0 
\]
and algebraic dimension $N > n$. These manifolds exhibit strong analogies to those with non-negative Ricci curvature: for instance, Corollaries \ref{cor_car_1} and \ref{cor_car_2} guarantee bounds on the Betti numbers of $M$ and further topological consequences.

\begin{remark}
	The inequality $\overline{\Ric}^N_f \ge 0$ was also pointed out in \cite[Corollary 2.3]{car_mon_tew}.
\end{remark}

\begin{remark}
	The criteria in Theorems \ref{teo_main_critical}, \ref{teo_main_critical_cs} and \ref{teo_main_critical_3} also relate to an elegant rigidity result by Castillon \cite{castillon} for surfaces satisfying
	\[
	-\Delta + \beta K \ge 0
	\]
for some $\beta \ge 0$, where $K$ is the Gaussian curvature. The  result was later extended by various authors, see in particular \cite{berard_castillon}.
\end{remark}

\vspace{0.2cm}

The paper is organized as follows: in Section \ref{sec_critical}, we recall the main results in criticality theory needed throughout the paper. In Section \ref{sec_criteria}, we prove the abstract criteria in Theorems \ref{teo_main_critical}, \ref{teo_main_critical_cs}, \ref{teo_main_critical_2} and \ref{teo_main_critical_3} and their corollaries. In Section \ref{sec_geometry}, we deduce our geometric consequences.

\begin{note}
	When we were completing this paper, we learned that the authors of \cite{apx} had independently obtained results that partially overlap with ours, with different techniques. We agreed with them to post the results on arXiv independently at the same time.
\end{note}

Inspecting the proof of the spectral splitting theorem in \cite{apx}, there seems to be an intriguing duality between the combined conformal method+criticality approach in the present paper and the $\mu$-bubble technique by Gromov \cite{gromov}. We think it could be interesting to investigate if this connection can be given a more rigorous formulation.

\section{Criticality theory for $-\Delta - V$}\label{sec_critical}

Let $M$ be a Riemannian manifold. Hereafter, an exhaustion $\{\Omega_j\}$ of $M$ is a collection of relatively compact open sets with smooth boundary satisfying
\[
\Omega_j \Subset \Omega_{j+1} \Subset M, \qquad \bigcup_{j=1}^\infty \Omega_j = M.
\]
Let $V \in L^\infty_\loc(M)$ and consider the operator
\[
L_V \doteq -\Delta - V.
\]
Assume that $L_V$ is non-negative in the spectral sense (we write $L_V\ge 0$), namely, that the associated quadratic form
\[
Q_V(\phi) \doteq (\phi, L_V\phi)_{L^2} = \int_M \Big[ |\nabla \phi|^2 - V \phi^2\Big]
\]
is non-negative for each $\phi \in \lip_c(M)$, the set of compactly supported Lipschitz functions.

\begin{notation}
Hereafter, given $V_1, V_2 \in L^\infty_\loc(M)$ we will say that $L_{V_1} \ge L_{V_2}$ whenever $Q_{V_1}(\phi) \ge Q_{V_2}(\phi)$ for each $\phi \in \lip_c(M)$. This is equivalent to say that $V_1 \le V_2$ a.e. on $M$.
\end{notation}

It is well-known, see for instance \cite[Lemma 3.10]{prs_book}, that $L_V \ge 0$ is equivalent to the existence of a positive weak solution $0< u \in H^1_\loc(M)$ to $L_V u \ge 0$, and also equivalent to the existence of a weak solution $0 < u \in C^1(M)$ to $L_Vu = 0$. The criticality theory discussed below, due to Murata \cite{murata} and Pinchover \& Tintarev \cite{pinchover, pinchover2, pincho_tinta}, describes the geometry of the cone of positive solutions to $L_V u \ge 0$, see also the work of Zhao \cite{zhao}.

\begin{definition}\label{def_weighted}
	For $V  \in L^\infty_\loc(M)$, define $L_V = -\Delta -V$ and let $\Omega \subseteq M$ be an open set.
	\begin{itemize}
		\item[-] $L_V$ is \textbf{subcritical} in $\Omega$ if there exists $w \ge 0$, $w\not\equiv 0$ in $\Omega$ (called a \textbf{Hardy weight}) such that
		\begin{equation}\label{114}
			\disp \int_\Omega w |\phi|^2 \le Q_V (\phi) \qquad \forall \ \phi \in \lip_c (\Omega).
		\end{equation}
		Otherwise, the operator $L_V$ is said to be \textbf{critical}. \vspace{0.1cm}
		\item[-] $L_V$ has a \textbf{weighted spectral gap} in $\Omega$ if there exists $W \in C(\Omega)$, $W>0$ on $\Omega$ such that
		\begin{equation}\label{114}
			\disp \int_\Omega W |\phi|^2 \le Q_V (\phi) \qquad \forall \ \phi \in \lip_c (\Omega).
		\end{equation}
		\vspace{0.1cm}
		\item[-] A sequence $\{\phi_j\} \in L^\infty_c(\Omega) \cap H^1(\Omega)$ is said to be a \textbf{null sequence} if  $\phi_j \ge 0$ a.e. for each $j$, $Q_V(\phi_j) \to 0$ as $j \to \infty$ and there exists a relatively compact open set $B\Subset M$ and $C>1$ such that $C^{-1} \le \|\phi_j\|_{L^2(B)} \le C$ for each $j$.
		\vspace{0.1cm}
		\item[-] A function $0 \le \eta \in H^1_\loc(\Omega)$, $\eta \ge 0$, $\eta \not \equiv 0$ is a \textbf{ground state} for $L_V$ on $\Omega$ if it is the $L^2_\loc(\Omega)$ limit of a null sequence.
	\end{itemize}
\end{definition}
We have the following fundamental result, known as the ground state alternative, \cite{murata,pincho_tinta}. The version stated below also includes some further implications, which can be found in \cite[Theorem 4.1]{bmr}.
\begin{theorem}\label{teo_alternative}
	Let $M$ be connected and non-compact, and consider an operator $L_V \ge 0$ with $V \in L^\infty_\loc(M)$.  Then, either $L_V$ has a weighted spectral gap or a ground state on $M$, and the two possibilities mutually exclude. Moreover, the following properties are equivalent:
	\begin{itemize}
		\item[$(i)$] $L_V$ is subcritical on $M$;
		\item[$(ii)$] $L_V$ has a weighted spectral gap on $M$;
		\item[$(iii)$] There exist two positive solutions $u_1,u_2 \in C(M) \cap H^1_\loc(M)$ of $L_V u \ge 0$ which are not proportional;
		\item[$(iv)$] for some (equivalently, each) $o \in M$, there exists a minimal positive distributional solution $G$ to $L_VG = \delta_o$;
		\item[$(v)$] For some (equivalently, any) $K \Subset M$ compact with non-empty interior, and for some (any) $0< \xi \in C(M) \cap H^1_\loc(M)$ solving  $L_V \xi \ge 0$, it holds
		$$
		\disp \capac_{V}(K, \xi) \doteq \disp  \inf_{\phi \in \mathscr{D}(K, \xi) } Q_V(\phi) > 0,
		$$
		where
		$$
		\mathscr{D}(K, \xi)  = \Big\{ \phi \in \lip_c(M) \ : \ \phi \ge \xi \ \text{ in } K\Big\}.
		$$
	\end{itemize}
	If $L_V$ has a ground state $\eta$, then:
	\begin{itemize}
		\item $\eta>0$ on $M$ and solves $L_V \eta =0$ (in particular, $\eta \in C^{1,\mu}_\loc(M)$). Moreover, every other positive solution $\xi$ to $L_V \xi \ge 0$ is a multiple of $\eta$, hence a ground state;
		\item there exists a positive function $W>0$ such that for each $\psi \in C^\infty_c(M)$ satisfying
	\[
	\int \psi \eta \neq 0,
	\]
	there exists a constant $C = C(M,W,\psi,V)$ such that
	\begin{equation}\label{eq_poincare}
	C^{-1} \int_M W \phi^2 \le Q_V(\phi) + C \left| \int_M \psi \phi\right|^2 \qquad \forall \, \phi \in \lip_c(M).
	\end{equation}
	\end{itemize}
	%
\end{theorem}

\begin{proof}[Proof (recap)]
	All the statements, apart from those regarding property $(iv)$, are contained in  \cite[Theorem 4.1]{bmr}. The equivalence $(i) \Leftrightarrow (iv)$ is the content of \cite[Theorem 2.4]{murata}.
	Inequality \eqref{eq_poincare} was proved in \cite{pincho_tinta,pincho_tinta_p}, see also \cite{weidl} for a related result.
\end{proof}

\begin{remark}\label{rem_hardy}
	Theorem \ref{teo_alternative} extends, to operators with a potential, the dichotomy between pa\-ra\-bolic and non-parabolic manifolds, which are exactly the manifolds for which $-\Delta$ is critical and subcritical, respectively. If $-\Delta$ is subcritical, by \cite{liwang_ens} a Hardy weight can be constructed from any non-constant positive solution $u$ of $\Delta u \le 0$ as
	\[
	w = \frac{|\nabla u|^2}{4u^2}.
	\]
	In particular, if $u$ is the Green kernel of $-\Delta$ on $\R^n$, $n \ge 3$ and $r$ is the distance to a fixed point, we recover the well-known Hardy weight
	\begin{equation}\label{eq_hardy_Rn}
	\frac{(n-2)^2}{4r^2},
	\end{equation}
	As shown in \cite{carron_hardy}, the same Hardy weight also occurs for minimal submanifolds $M^n \to N^p$ in a Cartan \& Hadamard manifold, where now $r$ is the restriction to $M$ of the extrinsic distance from a fixed origin of $N$. More general Hardy weights, including sharpened ones for minimal submanifolds in hyperbolic space, can be found in \cite[Section 5]{bmr} and \cite[Example 1.8]{liwang_ens}, see also \cite{bmr_mem} for geometric applications.
\end{remark}

\begin{remark}
	As a direct application of the ground state alternative, if $M$ satisfies a weighted Sobolev inequality
	\[
	\left( \int_M \eta |\phi|^{\frac{2\nu }{\nu -2}} \right)^{\frac{\nu -2}{\nu}} \le \So \int |\nabla \phi|^2 \qquad \forall \, \phi \in \lip_c(M)
	\]
	for some $0 < \eta \in C(M)$, $\So > 0$ and $\nu > 2$, then $-\Delta$ is subcritical.
\end{remark}

\begin{remark}
	We stress that Theorem \ref{teo_alternative} also holds for homogeneous, $p$-Laplace operators
	\[
	-\diver\left( |\nabla u|^{p-2}\nabla u\right) - V |u|^{p-2} u,
	\]
	see \cite{pincho_tinta_p,bmr}. Note that for $p \neq 2$ one cannot avail of the Doob transform to reduce the result to the case $V \equiv 0$ up to adding suitable weights to the Laplacian. See \cite[Section 2.3]{car} for more information on the Doob transform.
\end{remark}

The next strict convexity property, which will be crucial for us in what follows, was shown in \cite[Theorem 3.1]{pinchover2} in the Euclidean setting. Alternative proofs can be found in \cite[Proposition 4.3]{pincho_tinta_p} and \cite[Corollary 4.17]{pincho_psa}. Both arguments hold verbatim on manifolds, and we reproduce the one in \cite{pincho_psa} for the sake of completeness.

\begin{theorem}[\cite{pinchover2, pincho_tinta_p,pincho_psa}]\label{teo_subcrit}
	Let $V_0,V_1 \in L^\infty_\loc(M)$ and assume that $L_{V_0} \ge 0$, $L_{V_1} \ge 0$. Then, setting $V_t = (1-t)V_0 + tV_1$ for $t \in [0,1]$, it holds $L_{V_t} \ge 0$. Moreover, if $V_0$ does not coincide with $V_1$ a.e., for each $t \in (0,1)$ the operator $L_{V_t}$ is subcritical on $M$. In other words,
	\[
	\mathscr{K} \doteq \{V \in L^\infty_\loc(M) : L_V \ge 0\}
	\]
	is a convex set whose extremal points are those $V$ for which $L_V$ is critical.
\end{theorem}

\begin{proof}
	Since
	\begin{equation}\label{eq_inte_crit}
		Q_{V_t}(\phi) = (1-t)Q_{V_0}(\phi) + t Q_{V_1}(\phi),
	\end{equation}
	the assertion $L_{V_t} \ge 0$ is immediate. Moreover, if any of $L_{V_0},L_{V_1}$ is subcritical, say $L_{V_1}$, and $W_1$ is a Hardy weight, $L_{V_t}$ is subcritical for $t \in (0,1)$ with Hardy weight $tW_1$. We are left to consider the case $V_0 \not \equiv V_1$ and $L_{V_0},L_{V_1}$ critical. Suppose that $L_{V_t}$ is critical, and let $\eta$ be a ground state which is the limit of a null sequence $\{\phi_j\}$ for $Q_{V_t}$. From \eqref{eq_inte_crit}, if $t \in (0,1)$ then  $\{\phi_j\}$ is also a null sequence for $Q_{V_0}$ and $Q_{V_1}$, so by Theorem \ref{teo_alternative} $\eta$ is a positive solution to both
	\[
	\Delta \eta + V_0 \eta = 0 \qquad \text{and} \qquad \Delta \eta + V_1 \eta = 0.
	\]
	Hence, $V_0 \equiv V_1$.\\
	The above argument proves that weights $V$ for which $L_V$ is critical are extremals. Viceversa, if $L_V$ is subcritical pick a Hardy weight $W > 0$ for $V$, so $L_{V+W} \ge 0$. Since clearly $L_{V-W} \ge 0$, from
	\[
	V = \frac{1}{2} \left( (V+W) + (V-W) \right) 
	\]
	we obtain that $V$ is not extremal. 
	%
\end{proof}

The notion of criticality can be localized on each end of $M$ as follows: assume that $L \doteq -\Delta - V$ satisfies $L \ge 0$ on $M$. Hereafter, a pair $(K,\xi)$ is the data of a solution $0< \xi \in C^{1,\alpha}_\loc(M)$ to $L \xi = 0$ (equality here is important) and a compact set $K$ with non-empty interior and smooth boundary. Let $E$ be an end of $M$ with respect to $K$, that is, a connected component of $M \backslash K$ with non-compact closure.
%
Fix an exhaustion $\{\Omega_j\}$ of $M$ with $K \subset \Omega_1$, set $E_j = E \cap \Omega_j$ and consider the family of solutions $u_j$ to
\begin{equation}\label{def_uj_end}
\left\{\begin{array}{l}
	Lu_j = 0 \qquad \text{on } \, E_j, \\[0.2cm]
	u_j = \xi \ \ \text{ on } \, \partial E, \qquad u_j = 0 \ \ \text{ on } \, E \cap \partial \Omega_j.
\end{array}
\right.
\end{equation}
The maximum and comparison principles imply that $0 < u_j \le u_{j+1} \le \xi$, thus by elliptic estimates $u_j \uparrow u_\xi$ in $C^1_\loc(\overline{E})$ for some $u_\xi>0$ solving
\[
\left\{\begin{array}{l}
	Lu_\xi = 0 \qquad \text{on } \, E, \\[0.2cm]
	u_\xi = \xi \ \ \text{ on } \, \partial E, \qquad 0 < u_\xi \le \xi \ \ \text{ on } \, E.
\end{array}
\right.
\]
By comparison, $u_\xi$ does not depend on the chosen exhaustion. We call $u_\xi$ the minimal solution on $E$ with respect to $\xi$. The strong maximum principle implies that either $u_\xi \equiv \xi$ or $u_\xi < \xi$ on $E$.
\begin{definition}
	We say that $L$ is $\xi$-critical on $E$ if $u_\xi \equiv \xi$ on $E$, and $\xi$-subcritical otherwise.
\end{definition}
%
%
We have the following
\begin{theorem}\label{teo_subcrit_ends}
	Let $(M,g)$ be a Riemannian manifold, $V \in L^\infty_\loc(M)$ and assume that $L \doteq -\Delta - V \ge 0$. Then, the following are equivalent:
	\begin{itemize}
		\item[(i)] $L$ is critical on $M$;
		\item[(ii)] for every pair $(K,\xi)$, $L$ is $\xi$-critical on each end $E$ with respect to $K$;
		\item[(iii)] for some pair $(K,\xi)$, $L$ is $\xi$-critical on each end $E$ with respect to $K$;
	\end{itemize}
\end{theorem}
\begin{proof}
	We introduce the following construction. Fix a pair $(K,\xi)$, and let $E^{(1)},\ldots, E^{(k)}$ be the ends with respect to $K$. Up to attaching to $K$ the connected components of $M \backslash K$ with compact closure, we can assume that $M \backslash K = E^{(1)} \cup \ldots \cup E^{(k)}$. For each $i$, let $u^{(i)}$ be the minimal solution on $E^{(i)}$ with respect to $\xi$. Let $\{\Omega_j\}$ be a smooth exhaustion of $M$, and for each $i$ let $\{u^{(i)}_j\}$ be the sequence constructed in \eqref{def_uj_end} on $E^{(i)}$. Define
	\[
	\xi_j \doteq \left\{ \begin{array}{ll}
		\xi & \quad \text{on } \, K, \\
		u^{(i)}_j & \quad \text{on } \, E^{(i)} \cap \Omega_j.
	\end{array}\right.
	\]
	Then, $0 < \xi_j \le \xi$ on $\Omega_j$. Moreover, integrating by parts $0 = \xi L\xi$ on $K$ and $0 = u^{(i)}_j L u^{(i)}_j$ on $E^{(i)} \cap \Omega_j$ and subtracting the resulting identities we deduce
	\[
	Q_V(\xi_j) = \sum_{i=1}^k \int_{\partial E^{(i)}} \xi\left[ \partial_\eta \xi - \partial_\eta u^{(i)}_j\right].
	\]
	Moreover,  	
	\[
	\xi_j \uparrow \hat\xi \doteq \left\{ \begin{array}{ll}
		\xi & \quad \text{on } \, K, \\
		u^{(i)} & \quad \text{on } \, E^{(i)},
	\end{array}\right.
	\]
	and from $u^{(i)} \le \xi$ one easily deduces that $L\hat\xi \ge 0$ in the weak sense on $M$. We are ready to prove our equivalence.\\
	$(i) \Rightarrow (ii)$. By contradiction, assume that for some pair $(K,\xi)$ the operator $L$ is subcritical on some end. Then, by construction $\hat \xi \not \equiv \xi$. Since $\hat\xi = \xi$ on $K$, $\xi$ and $\hat\xi$ are distinct supersolutions for $L$ which are not proportional. By $(iii)$ in Theorem \ref{teo_alternative} we deduce that $L$ is subcritical on $M$, contradiction.\\
	$(ii)\Rightarrow (iii)$ is obvious.\\
	$(iii)\Rightarrow (i)$. Fix the pair $(K,\xi)$ in $(iii)$, and assume by contradiction that $L$ is subcritical on $M$. Then, the energy $Q_V(\xi_j)$ cannot vanish as $j \to \infty$ (otherwise, $\{\xi_j\}$ would be a null sequence). Since
	\[
	Q_V(\xi_j) \to \sum_{i=1}^k \int_{\partial E^{(i)}} \xi\left[ \partial_\eta \xi - \partial_\eta u^{(i)}\right],
	\]
we deduce that $u^{(i)} \neq \xi$ for some $i$, so $L$ is $\xi$-subcritical on $E^{(i)}$, contradiction.	
\end{proof}

%
\

\section{The conformal method, criticality and splitting criteria}\label{sec_criteria}

Let us consider a conformal deformation
\[
\bar g = e^{2\varphi}g, \qquad \varphi \in C^\infty(M),
\]
and denote with a bar superscript quantities in the metric $\bar g$. The Ricci curvatures of $g$ and $\bar g$ relate as follows:
\begin{equation}\label{eq_rel_ricci}
\overline{\Ric} = \Ric -(n-2)(\nabla^2 \varphi - \di \varphi \otimes \di \varphi) - \left[ \Delta \varphi + (n-2)|\di \varphi|^2 \right] g.
\end{equation}
Let $\gamma$ be a $\bar g$-geodesic, namely, a geodesic in the metric $\bar g$. Denoting by $s$ and $\bar s$, respectively, the $g$-arclength and the $\bar g$-arclength of $\gamma$, we have $\di \bar s = e^\varphi \di s$ and $\partial_{\bar s} = e^{-\varphi}\partial_s$, where $\varphi$ is evaluated at $\gamma(s)$. Since the $g$ and $\bar g$-unit tangent vectors to $\gamma$ satisfy $\gamma_{\bar s} = e^{-\varphi}\gamma_s$, we have
\[
0 = \bar \nabla_{\gamma_{\bar s}} \gamma_{\bar s} = e^{-2\varphi} \left( \nabla_{\gamma_s} \gamma_s + \gamma_s(\varphi)\gamma_s - \nabla \varphi\right)
\]
Therefore,
\begin{equation}\label{eq_link_phiss}
\begin{array}{lcl}
	(\varphi \circ \gamma)_{ss} & = & \nabla^2 \varphi(\gamma_s,\gamma_s) + \di \varphi(\nabla_{\gamma_s}\gamma_s) \\[0.3cm]
	& = & \nabla^2 \varphi(\gamma_s,\gamma_s) + \di \varphi \Big(\nabla \varphi - \gamma_s(\varphi)\gamma_s\Big) = \nabla^2 \varphi(\gamma_s,\gamma_s) + |\di \varphi|^2 - \di \varphi(\gamma_s)^2
\end{array}
\end{equation}
Therefore, when restricted to $\gamma$, \eqref{eq_rel_ricci} implies the remarkable identity
\begin{equation}\label{eq_remarkable}
\overline{\Ric}(\gamma_s,\gamma_s) = \Ric(\gamma_s,\gamma_s) - (n-2)(\varphi \circ \gamma)_{ss} - \Delta \varphi,
\end{equation}
appearing implicitly in \cite{shen_ye_min} and proved in the Appendix of \cite{enr}. Hereafter, the composition with $\gamma$ will tacitly be assumed, thus  we will write $\varphi_s,\varphi_{ss}$ rather than $(\varphi \circ \gamma)_s, (\varphi \circ \gamma)_{ss}$ and so forth.

We examine more closely the metric properties of the metric $\bar g = e^{2\varphi}g$. If $f,\psi : M \to \R$ and smooth functions, note that
\[
\begin{array}{lcl}
	\bar \nabla^2 f & = & \disp \nabla^2 f - (\di f \otimes \di \varphi + \di \varphi \otimes \di f) + g(\di f, \di \varphi)g, \\[0.4cm]
	\bar \Delta \psi & = & e^{-2\varphi}\Big( \Delta \psi + (n-2)g(\di \varphi, \di \psi)\Big)
\end{array}\]
Hence, from \eqref{eq_rel_ricci}, choosing
\[
f = (n-2)\varphi
\]
the $\infty$-Bakry-Emery Ricci curvature
\[
\overline{\Ric}_f \doteq \overline{\Ric} + \bar \nabla^2 f
\]
satisfies:
\begin{equation}\label{eq_ric_weight}
\overline{\Ric}_f = \Ric - (n-2) \di \varphi \otimes \di \varphi - (\Delta \varphi)g.
\end{equation}
Moreover, the associated weighted Laplacian
\[
\bar{\Delta}_f \psi \doteq \bar{\Delta} \psi - \bar g(\di f, \di \psi)
\]
of a function $\psi$ satisfies
\begin{equation}\label{eq_Lapl_weighted}
\bar{\Delta}_f \psi = e^{-2\varphi}\Delta \psi.
\end{equation}

Observe that, in dimension $n=2$, formulas \eqref{eq_ric_weight} and \eqref{eq_Lapl_weighted} become the well-known identities relating the Gaussian curvature and the Laplacian of conformally deformed metrics. The main computation, repeatedly used throughout the paper and essentially due to \cite{shen_ye_min}, is the following. Hereafter, we write
	\[
	C^k_0([a,b]) = \big\{ \eta \in C^k([a,b]) \ : \eta(a)=\eta(b) = 0\big\}.
	\]

\begin{proposition}\label{prop_basiccomp}
Assume that $(M^n,g)$ is a Riemannian manifold, let $V \in C(M)$ and define $L \doteq -\Delta - V$. Fix $0 < u \in C^2(M)$, and for $\beta > 0$ consider the metric
\[
\bar g = u^{2\beta} g.
\]
Assume that $\gamma : [a,b] \to M$ is a curve parametrized by $g$-arclength $s$, and assume that $\gamma$ is $\bar g$-length minimizing. Then, for each $\eta \in C^1_0([a,b])$ it holds
	\begin{equation}\label{eq_fundamental}
\begin{array}{l}
	\disp \beta \int_{a}^{b} \frac{Lu}{u}\frac{\eta^2}{u^{\beta}}\di s + \int_{a}^{b}  \Big(\Ric(\gamma_{s},\gamma_{s}) + \beta V\Big)\frac{\eta^2}{u^{\beta}}\di s + \beta \int_a^b |\nabla^\perp \log u|^2 \frac{\eta^2}{u^{\beta}} \di s \\[0.5cm]
	\qquad \le \disp (n-1)\int_{a}^{b} \frac{(\eta_{s})^2}{\eta^2} \frac{\eta^2}{u^{\beta}}\di s - \beta \int_{a}^{b}\frac{(u_s)^2}{u^2} \frac{\eta^2}{u^{\beta}}\di s - \beta(n-2)\int_{a}^{b} \frac{u_s}{u} \left(\frac{\eta^2}{u^{\beta}}\right)_s \di s,
\end{array}
\end{equation}
where $\nabla^\perp$ is the component of the $g$-gradient $\nabla$ orthogonal to $\gamma_s$.
\end{proposition}

\begin{proof}
	Setting $\varphi = \beta \log u$, we deduce from \eqref{eq_remarkable} the identity
	\[
	\overline{\Ric}(\gamma_s,\gamma_s) = \Ric(\gamma_s,\gamma_s) - (n-2)\beta (\log u)_{ss} + \beta \frac{Lu}{u} + \beta V + \beta |\di \log u|^2.
	\]
	Let $\bar s$ denote the $\bar g$-arclength. Since $\gamma$ is $\bar g$-minimizing, the second variation formula in the metric $\bar g$ yields
	\begin{equation}\label{eq_secondvar_nice_2}
		\begin{array}{lcl}
			0 & \le & \disp \int_{\bar s(a)}^{\bar s(b)} \Big[(n-1)(\eta_{\bar s})^2 - \overline{\Ric}(\gamma_{\bar s},\gamma_{\bar s})\eta^2 \Big] \di \bar s \\[0.5cm]
			& = & \disp \int_{a}^{b} \Big[(n-1)(\eta_{s})^2 - \overline{\Ric}(\gamma_{s},\gamma_{s})\eta^2 \Big] u^{-\beta}\di s \\[0.5cm]
			& = & \disp (n-1)\int_{a}^{b} (\eta_{s})^2u^{-\beta}\di s - \int_{a}^{b}  \Big(\Ric(\gamma_{s},\gamma_{s}) + \beta V\Big)\eta^2 u^{-\beta}\di s \\[0.5cm]
			& & \disp - \beta \int_{a}^{b} \Big(|\di\log u|^2 - (n-2)(\log u)_{ss}\Big) \eta^2 u^{-\beta}\di s - \beta \int_{a}^{b} \frac{Lu}{u}\eta^2u^{-\beta} \di s.
		\end{array}
	\end{equation}
	Using $|\di \log u|^2 = |\nabla^\perp \log u |^2 + (u_s)^2/u^2$ and integrating by parts the term with $(\log u)_{ss}$ we get \eqref{eq_fundamental}.
\end{proof}

The following proposition enables to construct a ``shortest line" $\gamma$ in a Riemannian manifold $(M,h)$ with more than one end, so that the completeness of $h$ is detected by the fact that $\gamma$ is defined on the whole of $\R$ once parametrized by $h$-arclength. The result improves on \cite[Lemma 1.1]{xucheng}.

\begin{proposition}\label{prop_shortest_line}
	Let $(M,h)$ be a Riemannian manifold containing a compact set $K$ so that $M \backslash K$ has more than one end. Then, there exist $T,T^* \in (0,\infty]$ and a unit speed curve $\gamma : (-T,T^*) \to M$ which satisfies:
	\begin{itemize}
		\item[(i)] $\gamma$ is length minimizing for every pair of its points, and $\gamma \cap K \neq \emptyset$.
		\item[(ii)] $\gamma$ is a divergent curve, i.e. $\gamma^{-1}(K')$ is compact for each $K' \subset M$ compact. Moreover, for $\eps$ small enough $\gamma_{[T^*-\eps,T^*)}$ and $\gamma_{(-T,-T+\eps]}$ belong to different ends of $M$ with respect to $K$.
		\item[(iii)] $(M,h)$ is complete if and only if both the curves $\gamma_{[0,T^*)}$ and $\gamma_{(-T,0]}$ have infinite length, that is, $T = T^* = \infty$.
	\end{itemize}
\end{proposition}

\begin{proof}
To construct $\gamma$, 
up to including in $K$ the components of $M \backslash K$ with compact closure, we can assume that $M \backslash K= \coprod_{i=1}^\ell E_i$, with $E_i$ an end. Pick an exhaustion $\{\Omega_n\}$ of $M$ with $K \Subset \Omega_1$. For each $i,n$, define $E_i(n) = E_i \backslash \Omega_n$. Define
	\[
	\tau_n = \min \Big\{ \di\big(\partial E_i(n),\partial E_j(n)\big) : i,j \in \{1,\ldots, \ell\}, \ i \neq j. \Big\}.
	\]
Since $(\overline{\Omega}_n,h)$ is a compact manifold with boundary, there exists a segment (i.e. a unit speed, minimizing geodesic) $\gamma_n$ in $\overline{\Omega}_n$ realizing $\tau_n$, joining $\partial E_{i_n}(n)$ to $\partial E_{j_n}(n)$ for some indices $i_n \neq j_n$. By construction $\gamma_n$ intersects $K$, thus up to translating the arclength parameter we can assume that $\gamma : [-T_{n},T_{n}^*] \to \overline{\Omega}_n$ for some positive $T_{n},T_{n}^*$, and that
\[
\gamma_n(0) \in K, \qquad \gamma_n(-T_{n}) \in \partial E_{i_n}(n), \qquad \gamma_n(T_{n}^*) \in \partial E_{j_n}(n).
\]
By minimality $\gamma_n((-T_{n},T_{n}^*)) \subset \Omega_n$, hence $\gamma_n$ is a geodesic. Furthermore, by construction there exist indices $i,j$ such that $\gamma_n$ joins $\partial E_i(n)$ and $\partial E_j(n)$ for infinitely many $n$. Up to passing to a subsequence in $n$ (still labelled the same), the family $\{\gamma_n\}$ locally converges to a minimizing geodesic $\gamma : (-T,T^*) \to M$ which crosses every $\partial E_i(n)$ and $\partial E_j(n)$, say at instants $t_n$ and $t_n^*$, respectively, and satisfies $\gamma(0)= o \in K$. In particular, $\gamma$ satisfies $(i)$. Note that $t_n \to -T$ and $t_n^* \to T^*$, because $\{\partial E_i(n)\}$ and $\{\partial E_j(n)\}$ are divergent sequences of sets as $n \to \infty$. To prove $(ii)$, it is enough to show that $\gamma_{[0,T^*)}$ is divergent and eventually valued in $E_j$, the case of $\gamma_{(-T,0]}$ being analogous. By contradiction, assume that there exists a compact set $K'$ (which we can assume contains $o$ in its interior) and a sequence $\tau_n \to T^*$ so that $\gamma(\tau_n) \in K'$. Pick $n_0$ such that $K' \subset \Omega_{n_0}$. Then, for each $n \ge n_0$ we can estimate
\[
|t^*_n-\tau_n| = \di \big( \gamma(t^*_n), \gamma(\tau_n) \big) \ge \di(\partial E_j(n), K') \ge \di(\partial E_j(n_0), K').
\]
From $t^*_n,\tau_n \to T^*$ we infer that necessarily $T^* = \infty$. However, from
\[
\tau_n = \di(o, \gamma(\tau_n)) \le \mathrm{diam}(K')
\]
and letting $n \to \infty$ we get a contradiction. Thus, $\gamma_{[0,T*)}$ is divergent, and since $\gamma(t_n) \in \partial E_j(n)$ we conclude that $\gamma$ is eventually contained in $E_j$. This proves $(ii)$.\\
To show $(iii)$, assume that $(M,h)$ is complete. Then, there exist $R_n \to \infty$ such that $B_{R_n}(o) \subset \Omega_n$. Since $\gamma$ is minimizing, by the definition of $t_n, t_n^*$ it holds $t_n \le -R_n$, $t_n^* \ge R_n$, hence $T = T^* = \infty$.\\
Viceversa, assume that $T = T^* = \infty$. In particular, $T_n^* \to \infty$. Let $\sigma : [0,R) \to M$ be a maximally extended unit speed geodesic issuing from $o$. We claim that $R = \infty$, from which the completeness follows. Assume by contradiction that $R < \infty$. Then, by ODE theory $\sigma$ is a divergent curve, and we can fix $k \in \{1,\ldots, \ell\}$ so that $\sigma([R_0,R)) \subset E_k$ for some $R_0>0$. Choose an index between $i,j$ which is different from $k$, say $i$. For each $n$, let $s_n$ be the first instant such that $\sigma(s_n) \in \partial E_k(n)$. We claim that
\begin{equation}\label{ine_dist}
\di( \gamma_n(0), \partial E_j(n)) \le \di( \gamma_n(0), \partial E_k(n)).
\end{equation}
Otherwise, concatenating $\gamma_n : (-T_n,0] \to M$ to a segment $\eta$ from $\gamma_n(0)$ to $\partial E_k(n)$ we would get a curve joining $\partial E_i(n)$ to $\partial E_k(n)$ shorter than $\gamma_n$, contradiction. From \eqref{ine_dist} we deduce
\[
\begin{array}{lcl}
T_n^* & = & \disp \di( \gamma_n(0), \partial E_j(n)) \le \disp \di( \gamma_n(0), \partial E_k(n)) \\[0.2cm]
& \le & \disp \disp \di( \gamma_n(0), \gamma(0)) + \di( \gamma(0), \partial E_k(n)) \\[0.2cm]
& \le & \disp \di( \gamma_n(0), \gamma(0)) + s_n < \di( \gamma_n(0), \gamma(0)) + R.
\end{array}
\]
Letting $n \to \infty$ we get a contradiction.
\end{proof}

Here is our first main result:

\begin{theorem}\label{teo_main_critical}
	Let $(M,g)$ be a complete Riemannian manifold, and assume that there exists $V \in C^{0,\alpha}_\loc(M)$ such that
	\begin{equation}\label{eq_iporicci}
	\Ric \ge -\beta V g \qquad \text{on } \, M, \qquad L \doteq -\Delta - V \ge 0,
	\end{equation}
	where
	\[
	0<\beta \leq \frac{n-1}{n-2}.
	\]
	If either
	\begin{itemize}
		\item[(i)] $n=3$, or
		\item[(ii)] $n \ge 4$, and there exists a compact set $Z$ and a weak solution $0<w \in C(\overline{M\backslash Z})$ to $Lw \ge 0$ satisfying
		\begin{equation}\label{eq_growth_w}
		w(x) = o \big( r(x) \big)^{\frac{n-1}{\beta(n-3)}} \qquad \text{as } \, r(x) \to \infty,
		\end{equation}
	\end{itemize}
	where $r(x)$ is the distance from a fixed origin, then the following holds: either
	\begin{itemize}
		\item $M$ has only one end, or
		\item $L$ is critical and, for each $x \in M$,
		\begin{equation}\label{eq_lowereige_Ric}
			\min_{X \in T^1_xM} \left( \Ric(X,X) + \beta V(x)\right) =0.
		\end{equation}
	\end{itemize}
\end{theorem}

\begin{proof}
	Assume that $M$ has at least two ends, and, by contradiction, that $L$ is subcritical. Pick an exhaustion $\{\Omega_j\}$ of $M$ with $Z \subset \Omega_1$ and so that $M \backslash \Omega_1$ has at least two ends $E_1$ and $E_2$. By $(i) \Longleftrightarrow (ii)$ in Theorem \ref{teo_alternative}, there exists $0 \le W \in C^\infty_c(\Omega_3)$ which is positive on $\overline{\Omega}_2$ and satisfies
	\[
\disp \int_\Omega 2W |\eta|^2 \di x \le Q_V (\eta) \qquad \forall \, \eta \in \lip_c(M).
\]
Thus, $L - W \ge 0$ and is subcritical. Picking a point $o \in \Omega_1$, by $(iv)$ in Theorem \ref{teo_alternative} there exists a minimal positive solution $G \in C^{2,\alpha}_\loc(M \backslash \{o\})$ to $LG - WG = \delta_o$. Fix also a global solution $0 < v \in C^{2,\alpha}_\loc(M)$ to $Lv = 0$ on $M$. Since $G$ is constructed as the limit of a sequence of kernels $G_j$ on $\Omega_j$ with Dirichlet boundary conditions, comparing $G_j$ to $w$ and $v$ on $\Omega_j \backslash \Omega_3$ (recall, $LG_j = 0$ there) and letting $j \to \infty$, there exists a constant $C>0$ such that $G \le C \min \big\{w,v\big\}$ on $M \backslash \Omega_3$. Thus, up to increasing $C$,
	\begin{equation}\label{eq_uw}
	G \le C \min \big\{w,v\big\} \qquad \text{on } \, M \backslash \Omega_1.
	\end{equation}
The function $z \doteq G/v$ therefore solves
\[
\left\{ \begin{array}{ll}
\diver\left( v^2 \nabla z\right) \le 0 & \quad \text{on } \, M \backslash\{o\} \\[0.2cm]
z(x) \to \infty & \quad \text{as } \, x \to o \\[0.2cm]
z \le C & \quad \text{on } \, M \backslash \Omega_1.
\end{array}\right.
\]
Consider a concave function $\eta \in C^2([0,\infty))$ satisfying $\eta(t) = t$ for $t \in [0,2C]$, $\eta(t) = 3C$ for $t \ge 4C$. Then, $\eta(z) \in C^2(M)$ solves
\[
\diver\left( v^2 \nabla \eta(z)\right) \le 0 \qquad \text{on } \, M,
\]
and thus
\[
u \doteq v \eta(z) \in C^2(M) \qquad \text{solves } \quad \left\{ \begin{array}{ll}
	L u \ge 0 & \quad \text{on } \, M \\[0.2cm]
	u = G & \quad \text{on } \, M \backslash \Omega_1.
\end{array}\right.
\]
Consider the conformal deformation
	\[
	\bar g = u^{2\beta} g.
	\]
We construct a $\bar g$-line $\gamma : (-T, T^*) \to M$ as in Proposition \ref{prop_shortest_line}. By property $(ii)$ of $\gamma$ and since $g$ is complete, it follows that once reparametrized by its  $g$-arclength $s$ the curve $\gamma$ is defined on the whole of $\R$. Also, because $M \backslash \Omega_1$ is disconnected and again by property $(ii)$ we deduce that $\gamma \cap \overline{\Omega}_1 \neq \emptyset$ and $\gamma \cap (\Omega_2 \backslash \overline\Omega_1) \neq \emptyset$. By Proposition \ref{prop_basiccomp}, for each $\eta \in C^1_c(\R)$ it holds
\[
	\begin{array}{l}
		\disp \beta \int_{-\infty}^{\infty} \frac{Lu}{u}\frac{\eta^2}{u^{\beta}} + \int_{-\infty}^{\infty}  \Big(\Ric(\gamma_{s},\gamma_{s}) + \beta V\Big)\frac{\eta^2}{u^{\beta}} \\[0.5cm]
		\qquad \le \disp (n-1)\int_{-\infty}^{\infty} \frac{(\eta_{s})^2}{\eta^2} \frac{\eta^2}{u^{\beta}} - \beta \int_{-\infty}^{\infty}\frac{(u_s)^2}{u^2} \frac{\eta^2}{u^{\beta}} - \beta(n-2)\int_{-\infty}^{\infty} \frac{u_s}{u} \left(\frac{\eta^2}{u^{\beta}}\right)_s,
	\end{array}
\]
where we omitted to write $\di s$ for the ease of notation. From \eqref{eq_iporicci},
	\[
	\begin{array}{lcl}
	\disp \beta\int_{-\infty}^\infty \frac{Lu}{u}\frac{\eta^2}{u^{\beta}} & \le & \disp(n-1) \int_{-\infty}^{\infty} \frac{(\eta_{s})^2}{\eta^2} \frac{\eta^2}{u^{\beta}} -\beta \int_{-\infty}^\infty \frac{(u_s)^2}{u^2} \frac{\eta^2}{u^{\beta}} - \beta(n-2)\int_{-\infty}^\infty \frac{u_s}{u} \left(\frac{\eta^2}{u^{\beta}}\right)_s.
	\end{array}
	\]
Write $\eta = u^\theta \psi$ for some $\theta \ge 0$ to be determined. Then,
	\[
\begin{array}{lcl}
	\disp \frac{\beta}{n-1}\int_{-\infty}^\infty \frac{Lu}{u}u^{2\theta-\beta}\psi^2  & \le & \disp \int_{-\infty}^{\infty} \left[ \theta^2 - \frac{2\theta\beta(n-2)}{n-1} +\frac{\beta^2(n-2)}{n-1} - \frac{\beta}{n-1}\right] \frac{(u_s)^2}{u^2}u^{2\theta-\beta}\psi^2\\[0.5cm]
	& & \disp + \int_{-\infty}^\infty (\psi_s)^2 u^{2\theta-\beta}  + 2\left[\theta-\frac{\beta(n-2)}{n-1}\right]  \int_{-\infty}^\infty \frac{u_s \psi\psi_s}{u} u^{2\theta-\beta}.
\end{array}
\]
Choosing $\theta=\frac{\beta(n-2)}{n-1}$, we conclude
\begin{equation}\label{eq_thekeyineq}
\frac{\beta}{n-1}\int_{-\infty}^\infty \frac{Lu}{u} u^{\frac{\beta(n-3)}{n-1}}\psi^2 \le  \int_{-\infty}^\infty (\psi_s)^2 u^{\frac{\beta(n-3)}{n-1}}
\end{equation}
Let $0\le \psi_R \le 1$ be an even function satisfying $\psi_R \equiv 1$ on $[-R,R]$, $\psi_R \in C^\infty_c((-2R,2R))$ and $|(\psi_R)_s| \le 4/R$. Then,
\begin{equation}\label{eq_cut}
\frac{\beta}{n-1} \int_{-R}^R \frac{Lu}{u} u^{\frac{\beta(n-3)}{n-1}} \le \frac{16}{R^2} \left[ \int_R^{2R} u^{\frac{\beta(n-3)}{n-1}} + \int_{-2R}^{-R} u^{\frac{\beta(n-3)}{n-1}}\right].
\end{equation}
If $n=3$, letting $R \to \infty$ we deduce $Lu \equiv 0$ along $\gamma$. Similarly, if $n \ge 4$, by \eqref{eq_uw} there exists $C$ such that
\begin{equation}\label{eq_uw_bis}
	u(\gamma(s)) \le C \big(1+w(\gamma(s))\big) \qquad \forall \, s \in \R.
\end{equation}
Assumption \eqref{eq_growth_w} (without loss of generality, we can assume that $r$ is the distance from $\gamma(0)$) and $|\gamma_s| = 1$ imply
\[
r\big(\gamma(s)\big) \le |s|,
\]
whence
\[
u(\gamma(s)) = o\big( |s|^{\frac{n-1}{\beta(n-3)}} \big) \qquad \text{as } \, |s| \to \infty.
\]
Plugging into \eqref{eq_cut} and letting $R \to \infty$ we again deduce $Lu \equiv 0$ along $\gamma$.\\
To conclude, we observe that
\[
Lu = LG = WG > 0 \qquad \text{on } \, \gamma \cap (\Omega_2 \backslash \overline{\Omega}_1) \neq \emptyset,
\]
leading to the desired contradiction. To prove \eqref{eq_lowereige_Ric}, assume by contradiction that
	\[
	\Ric + \beta Vg > 0 \qquad \text{at some } \, x_0 \in M,
	\]
	and fix $0 \le w \in C^\infty(M)$ not identically zero satisfying
	\[
	\Ric + \beta (V-w) g \ge 0 \qquad \text{on } \, M.
	\]
	Since $L + w \ge L \ge 0$, and $w \not \equiv 0$, $L + w$ is subcritical. However, we can apply the first part of the proof with $V-w$ replacing $V$ to deduce that $L + w$ is critical, contradiction.
\end{proof}

\begin{example}\label{ex_catenoid}
	We denote by  $C^n \to \R^{n+1}$, $n\geq 3$, the higher-dimensional catenoid. Following  \cite{tam_zhou} (see also \cite{DoCD}), $C^n$ depends on a real parameter $\phi_0>0$ and is defined by the map
	 \begin{align*}
		F : (-S,S) \times \mathbb{S}^{n-1} &\rightarrow \R^{n+1}, \\ \pa{s, \omega} &\mapsto \pa{\phi(s)\omega, s},
	\end{align*}
	where
	\[
	S = \int_{\phi_0}^\infty \frac{\di t}{\sq{(t/\phi_0)^{2(n-1)}-1}^{1/2}} < \infty
	\]
	and $\phi$ solves
	\[
		\begin{cases}
			\phi'' \sq{1+\pa{\phi'}^2}^{-\frac32}- \frac{n-1}{\phi} \sq{1+\pa{\phi'}^2}^{-\frac12} = 0 \\[0.2cm]
			\phi(0) = \phi_0 >0 \\
			\phi'(0) = 0.
		\end{cases}
	\]
By \cite{tam_zhou}, the norm of the second fundamental $|A|$ is positive on $C^n$, globally bounded (in fact, $C^n$ has finite total curvature) and solves
	\[
	\Delta |A|^{\frac{n-2}{n}} + \frac{n-2}{n} |A|^2 |A|^{\frac{n-2}{n}} = 0.
	\]
	Therefore, $C^n$ is $\frac{n-2}{n}$-stable:
	\[
	L \doteq -\Delta - V \ge 0, \qquad \text{where } \, V \doteq \frac{n-2}{n}|A|^2.
	\]
	By Gauss equation and the refined Kato inequality,
	\[
	\Ric \ge - \frac{n-1}{n}|A|^2 = \frac{n-1}{n-2} V,
	\]
	Hence, we can apply Theorem \ref{teo_main_critical} with $w = |A|^{\frac{n-2}{n}}$ to deduce that $L$ is critical on $C^n$.
\end{example}

We next examine the case where $V_+$ has compact support, namely, $\Ric \ge 0$ outside of a compact set. In this case, we can drop the dimensional restriction and get:

\begin{theorem}\label{teo_main_critical_cs}
	Let $(M,g)$ be a complete Riemannian manifold, and assume that there exists $V \in C^{0,\alpha}_\loc(M)$ with $V_+$ compactly supported such that
	\begin{equation}\label{eq_assu_cs}
	\Ric \ge - \frac{n-1}{n-2}V g \qquad \text{on } \, M, \qquad L \doteq -\Delta - V \ge 0.
	\end{equation}
	then, either
	\begin{itemize}
		\item $M$ has only one end, or
		\item $L$ is critical and, for each $x \in M$,
		\[
		\min_{X \in T^1_xM} \left( \Ric(X,X) + \frac{n-1}{n-2}V(x)\right) =0.
		\]
	\end{itemize}
\end{theorem}

\begin{proof}
	Assuming, by contradiction, that $M$ has at least two ends and that $L$ is subcritical, we follow the proof of Theorem \ref{teo_main_critical} and construct $u,\gamma$. For $\tau \ge 1$, let $\xi : [0,\infty) \to \R$ be a concave, non-decreasing function satisfying
	\[
	\xi(x) = x \quad \text{for } \, x \in [0,1], \qquad \xi(x) = 2 \quad \text{for } \, x \ge 4, \qquad \xi'(x) \le 1 \ \ \forall \, x \in [0,\infty)
	\]
	and define
	\[
	\xi_\tau(x) = \tau \xi (x/\tau), \qquad u_\tau = \xi_\tau(u)
	\]	
	Then, $0 \le \xi_\tau' \le 1$, $\xi_\tau - x\xi_\tau' \ge 0$, $u_\tau \in L^\infty(M)$ and $\xi_\tau' \uparrow 1$, $u_\tau \uparrow u$ as $\tau \to \infty$. Moreover, 	
	\[	
	-L u_\tau \le \xi_\tau'(u)Lu + V(u_\tau - u \xi_\tau'(u)) \le \xi_\tau'(u)Lu + V_+(u_\tau - u \xi_\tau'(u))
	\]
	By using $\beta = \frac{n-1}{n-2}$ and $u_\tau$ in place of $u$, and proceeding as in Theorem \ref{teo_main_critical}, we obtain the following analogue of \eqref{eq_thekeyineq}:
	\begin{equation}\label{eq_good_compact}
		\frac{1}{n-2}\int_{-\infty}^\infty \frac{\xi_\tau'(u)Lu}{u_\tau} u_\tau^{\frac{n-3}{n-2}}\psi^2 - \frac{1}{n-2}\int_{-\infty}^\infty \frac{V_+(u_\tau-u\xi_\tau'(u))}{u_\tau} u_\tau^{\frac{n-3}{n-2}}\psi^2 \le  \int_{-\infty}^\infty (\psi_s)^2 u_\tau^{\frac{n-3}{n-2}}
	\end{equation}
	Consider the family of cutoffs $\psi_R$ as above, so  $\psi_R \uparrow 1$ pointwise as $R \to \infty$. Taking into account that the first integral in the left-hand side of \eqref{eq_good_compact} is non-negative and the second is restricted to $\{ s \in \R : V_+(\gamma(s)) \neq 0\}$, a compact set since $\gamma$ diverges, letting $R \to \infty$ and applying monotone and Lebesgue convergence theorems we get
	\[
	\int_{-\infty}^\infty \frac{\xi_\tau'(u)Lu}{u_\tau} u_\tau^{\frac{n-3}{n-2}} - \int_{-\infty}^\infty \frac{V_+(u_\tau-u\xi_\tau'(u))}{u_\tau} u_\tau^{\frac{n-3}{n-2}} \le 0.
	\]
	In particular, the first integral is finite. Letting $\tau \to \infty$ and applying Fatou's Lemma (to the first integral) and Lebesgue's Theorem (to the second), we conclude
	\[
	\int_{-\infty}^\infty \frac{Lu}{u} u^{\frac{n-3}{n-2}} \le 0.
	\]
	thus $Lu \equiv 0$ along $\gamma$. The rest of the argument follows verbatim the one in Theorem \ref{teo_main_critical}.
\end{proof}

\begin{remark}
	Theorem \ref{teo_main_critical_cs} should be compared to \cite[Corollary 4.2]{liwang_ens}. There, the authors assume \eqref{eq_assu_cs} with $V \ge 0$ and the polynomial volume growth
	\[
	|B_r| \le C r^{2(n-1)} \quad \text{for } \, r \ge 1,
	\]
	a condition satisfied when $V_+$ is compactly supported because of Bishop-Gromov volume comparison. Applying \cite{liwang_ens} we deduce that, in the setting of Theorem \ref{teo_main_critical_cs} with $V \ge 0$, either $M$ has only one \emph{non-parabolic} end, or $M$ splits as Example \ref{ex_2_liwang} below.
\end{remark}


To further comment on the assumptions in Theorem \ref{teo_main_critical} and introduce our next result, denote by $\lambda_1(M)$ the bottom of the Laplace spectrum on $M$.

\begin{example}\label{ex_2_liwang}
	This example is taken from \cite[Proposition 6.1]{liwang_ens}. Given $n \ge 3$, a compact manifold $(P^{n-1},h)$ and a function $0 < \eta \in C^\infty(\R)$, consider the warped product
	\[
	M = \R \times P, \qquad g = \di t^2 + \eta(t)^2 h
	\]
	Under the assumptions
	\[
	\eta'' > 0, \qquad (n-2) (\log\eta)'' + \eta^{-2} \Ric_P \ge 0,
	\]
	the manifold $M$ satisfies
	\[
	\Ric \ge - \frac{n-1}{n-2} V g \qquad \text{with} \qquad V = (n-2) \frac{\eta''}{\eta}.
	\]
	The function
	\[
	v = \int_0^t \frac{\di s}{\eta(s)^{n-1}}
	\]	
	is harmonic and $w = |\nabla v|^{\frac{n-2}{n-1}} = c \cdot \eta(t)^{2-n}$ is a positive solution to
	\[
	Lw \doteq -\Delta w - Vw = 0, \qquad \text{so } \, L \ge 0.
	\]
	Therefore, if $\eta$ is bounded from below by a positive constant on $\R$ and $P$ is compact, by Theorem \ref{teo_main_critical} the operator $L$ is critical. An example is given by the choices
	\[
	P \ \text{ compact with } \ \Ric_P \ge -(n-2)h, \qquad \eta(t) = \ch t
	\]
	considered in \cite{liwang_positive_1}. In this case, the resulting warped product metric $\di t^2 + (\ch t)^2h$ satisfies
	\[
	\Ric \ge -(n-1)g, \qquad -\Delta - (n-2) \ge 0 \quad \text{and is critical on $M$}.
	\]
	In particular, $\lambda_1(M) = n-2$.
\end{example}

\begin{example}\label{ex_1_liwang}
The following example shows that a growth condition on $w$ in Theorem \ref{teo_main_critical} is necessary to prove criticality in dimension $n \ge 4$. Consider in Example \ref{ex_2_liwang} the choices
\[
P \ \text{ compact with } \ \Ric_P \ge 0, \qquad \eta(t) = e^t,
\]
see \cite[Example 2.2]{liwang_positive_1}. The resulting warped product metric $\di t^2 + e^{2t}h$ satisfies
\[
\Ric \ge -(n-1)g, \qquad L \doteq -\Delta - (n-2) \ge 0.
\]
In particular, $\lambda_1(M) \ge n-2$. However, the positive solution $w = e^{(2-n)t}$ to $Lw =0$ singled out in Example \ref{ex_2_liwang} does not satisfy \eqref{eq_growth_w}, thus in dimension $n \ge 4$ Theorem \ref{teo_main_critical} is not applicable to guarantee that $L$ is critical. Indeed, we show that $L$ is subcritical on $M$. To see this, for functions $u$ of $t$ only, equation $\Delta u + (n-2)u=0$ becomes
\[
u''(t) + (n-1)u'(t) + (n-2)u(t) = 0,
\]
whose general solution is spanned by $\{e^{-(n-2)t}, e^{-t}\}$ in dimension $n \ge 4$ and by $\{e^{-t}, te^{-t}\}$ in dimension $3$. Thus, if $n \ge 4$ the operator $L$ admits the positive solutions $e^{-(n-2)t}$ and $e^{-t}$ which are not proportional, hence it is subcritical. On the other hand, if $n=3$ the only positive solution just depending on $t$ is, up to scaling, the function $e^{-t}$. In fact, Theorem \ref{teo_main_critical} guarantees that $L$ is critical.\\
Let us examine this example further. The function
\[
\hat w = e^{-\frac{n-1}{2}t} \qquad \text{solves} \qquad \Delta \hat w + \frac{(n-1)^2}{4} \hat w = 0;
\]
hence,
\begin{equation}\label{eq_spectrum}
\lambda_1(M) \ge \frac{(n-1)^2}{4}
\end{equation}
(in fact, equality holds in \eqref{eq_spectrum} by Cheng's eigenvalue estimate).  Equivalently, since $\Ric \ge -(n-1)$ this can be rewritten as follows:
\[
\Ric  \ge - \frac{4}{n-1} \hat{V}, \qquad \hat{L} \doteq -\Delta - \hat{V} \ge 0 \quad \text{with } \, \hat{V} = \frac{(n-1)^2}{4}.
\]
Note that $\frac{(n-1)^2}{4} \ge (n-2)$, with equality iff $n=3$. We claim that $\hat{L}$ is critical, and to see this we check each end separately.
Let $E^{(1)} = (-\infty,0) \times P$ and $E^{(2)} = (1,\infty) \times P$ be the ends with respect to $K = [0,1]\times P$. Consider $\hat w$ as a supersolution, and let $u^{(i)}$ be the minimal solution on $E^{(i)}$ with respect to $\hat w$. Then, by construction $u^{(i)} = \lim_{j \to \infty} u^{(i)}_j$ where
\[
\left\{ \begin{array}{ll}
	\hat{L} u^{(1)}_j = 0 & \text{on } \, [-j,0] \times P, \\[0.2cm]
	u^{(1)}_j = 1 & \text{ on } \, \{0\} \times P, \\[0.2cm]
	u^{(1)}_j = 0 & \text{ on } \, \{-j\} \times P.	
\end{array}
\right. \qquad
\left\{ \begin{array}{ll}
	\hat{L} u^{(2)}_j = 0 & \text{on } \, [1,j] \times P, \\[0.2cm]
	u^{(2)}_j = e^{-\frac{n-1}{2}} & \text{ on } \, \{1\} \times P, \\[0.2cm]
	u^{(2)}_j = 0  & \text{ on } \, \{j\} \times P.	
\end{array}
\right.
\]
By uniqueness of solutions, $u^{(i)}_j$ only depends on $t$ (since its average over $P$ is still a solution). The general solution to the ODE
\[
-\hat L u = u''(t) + (n-1)u'(t) + \frac{(n-1)^2}{4} u(t) = 0
\]
is $e^{-\frac{n-1}{2}t}(a + b t)$ with $a,b \in \R$, whence
\[
u^{(1)}_j(t) = \frac{j+t}{j} e^{-\frac{n-1}{2}t}, \qquad u^{(2)}_j(t) = \frac{j-t}{j-1}e^{-\frac{n-1}{2}t}.
\]
Letting $j \to \infty$ we get $u^{(1)} = \hat w$ and $u^{(2)} = \hat w$, therefore $\hat L$ is $\hat w$-critical on each end. By Theorem \ref{teo_subcrit_ends}, $\hat L$ is critical on $M$.
\end{example}


Example \ref{ex_1_liwang} may suggest that if
\[
	\Ric \ge - \frac{4}{n-1}V g \qquad \text{on } \, M, \qquad L \doteq -\Delta - V \ge 0,
\]
then the dichotomy in the conclusion of Theorem \ref{teo_main_critical} holds in any dimension $n \ge 3$ without the growth condition \eqref{eq_growth_w}. It is interesting, in this respect, to compare to \cite[Theorem C]{liwang_ens} and to the detailed ODE analysis in Sections 6 and 8 therein. In the next result, Theorem \ref{chegrointro} in the Introduction, we are able to obtain a stronger conclusion when $\beta < \frac{4}{n-1}$. We restate the result for the convenience of the reader.

\begin{theorem}\label{teo_main_critical_2}
	Let $(M,g)$ be a complete Riemannian manifold of dimension $n \ge 3$, and assume that there exists $V \in C^{0,\alpha}_\loc(M)$ such that
	\begin{equation}\label{eq_iporicci_2}
		\Ric \ge - \beta V g \qquad \text{on } \, M, \qquad L \doteq -\Delta - V \ge 0,
	\end{equation}
	where either
	\[
	\begin{array}{ll}
	\disp \quad \quad 0<\beta < \frac{4}{n-1}, & \quad \text{or }  \\[0.5cm]
	\disp \frac{4}{n-1} \le \beta < \frac{n-1}{n-2}  & \quad \text{and $V_+$ is  compactly supported.}
	\end{array}
	\]


	Then, either
	\begin{itemize}
		\item[(i)] $M$ has only one end, or
		\item[(ii)] $V \equiv 0$ and $M = \R \times P$ with the product metric, for some compact $P$ with $\Ric_P \ge 0$.	
	\end{itemize}
\end{theorem}

\begin{proof}
	We first examine the case $\beta \in \left(0, \frac{4}{n-1}\right)$. Let us choose $\sigma \in \left(\beta , \frac{4}{n-1}\right)$ and set
	\[
	\hat V \doteq \frac{\beta}{\sigma} V.
	\]
	Then, by Theorem \ref{teo_subcrit}
	\begin{equation}\label{eq_perturb}
	\Ric \ge - \sigma \hat V g, \qquad \hat L \doteq -\Delta - \hat V \ge 0.
	\end{equation}
	Moreover, since $-\Delta \ge 0$, $-\Delta - V \ge 0$ and $\beta/\sigma \in (0,1)$, by Theorem \ref{teo_subcrit} $\hat L$ is subcritical unless $\hat V \equiv 0$, that is, unless $V \equiv 0$. Let $0 < u \in C^{2,\alpha}_\loc(M)$ solve $\hat L u \ge 0$, and consider the conformal deformation $\bar g = u^{2\sigma}g$. If we assume that $M$ has at least two ends, using Proposition \ref{prop_shortest_line} we guarantee the existence of a $\bar g$-line $\gamma : (T, T^*) \to M$ in $(M,\bar g)$. By $(ii)$ in Proposition \ref{prop_shortest_line} and since $g$ is complete, it follows that once reparametrized by $g$-arclength $s$ the curve $\gamma$ is defined on the whole of $\R$. Proposition \ref{prop_basiccomp} implies
	\[		
	\begin{array}{l}
		\disp \sigma \int_{-\infty}^{\infty} \frac{\hat Lu}{u}\frac{\eta^2}{u^{\sigma}} + \int_{-\infty}^{\infty}  \Big(\Ric(\gamma_{s},\gamma_{s}) + \sigma \hat V\Big)\frac{\eta^2}{u^{\sigma}}  \\[0.5cm]
		\qquad \le \disp (n-1)\int_{-\infty}^{\infty} \frac{(\eta_{s})^2}{\eta^2} \frac{\eta^2}{u^{\sigma}} - \sigma \int_{-\infty}^{\infty}\frac{(u_s)^2}{u^2} \frac{\eta^2}{u^{\sigma}} - \sigma(n-2)\int_{-\infty}^{\infty} \frac{u_s}{u} \left(\frac{\eta^2}{u^{\sigma}}\right)_s,
	\end{array}
	\]
	holds for every $\eta \in C^1_c(\R)$. Writing $\eta = u^{\sigma/2} \psi$ and rearranging we get
	\[
	\begin{array}{l}
		\disp \sigma \int_{-\infty}^\infty \frac{\hat Lu}{u}\psi^2 + \int_{-\infty}^\infty  \Big(\Ric(\gamma_{s},\gamma_{s}) + \sigma \hat V\Big)\psi^2 \\[0.5cm]
		\qquad \le \disp \sigma\left[ \frac{(n-1)\sigma}{4}-1\right] \int_{-\infty}^{\infty} \frac{(u_s)^2}{u^2}\psi^2 - \sigma(n-3)\int_{-\infty}^\infty \frac{u_s \psi\psi_s}{u} + (n-1)\int_{-\infty}^\infty (\psi_s)^2.
	\end{array}
	\]
	By \eqref{eq_perturb} and since $\sigma < \frac{4}{n-1}$, applying Young's inequality to the second term of the right-hand side we deduce the inequality
	\[
	\int_{-\infty}^\infty \frac{\hat Lu}{u}\psi^2 \le C \int_{-\infty}^\infty (\psi_s)^2
	\]	
	for some constant $C>0$. By choosing an even cut-off $0\le \psi_R \le 1$ such that $\psi_R \equiv 1$ on $[-R,R]$, $\psi_R \in C^\infty_c((-2R,2R))$ and $|(\psi_R)_s| \le 4/R$, letting $R \to \infty$ we deduce that $\hat Lu=0$ along $\gamma$. As $u$ can be any supersolution, we actually proved that any $C^{2,\alpha}_\loc$ supersolution $u$ satisfy $\hat Lu = 0$ on some curve. If $\hat L$ were subcritical, then by Theorem \ref{teo_alternative} there would exist $0 < W \in C(M)$ such that $\hat L - W \ge 0$. By approximating $W$ from below, we can assume that $W \in C^\infty(M)$. It would therefore exist a positive solution $u \in C^{2,\alpha}_\loc(M)$ to $\hat L u = Wu > 0$ on the entire $M$, contradiction. We conclude that $\hat L$ is critical, thus $V \equiv 0$. The conclusion in $(ii)$ thus follows from Cheeger \& Gromoll's splitting theorem \cite{chgr}.\\[0.2cm]
	If $\beta \in \left[ \frac{4}{n-1}, \frac{n-1}{n-2}\right)$ and $V_+$ is compactly supported, fix $\sigma \in \left(\beta, \frac{n-1}{n-2}\right)$ and $\hat V, \hat L$ as above. We follow the proof of Theorem \ref{teo_main_critical_cs} with $\sigma$ replacing $\frac{n-1}{n-2}$ up to \eqref{eq_good_compact}, which in view of \eqref{eq_cut} now becomes
		\[
		\frac{\sigma}{n-1}\int_{-\infty}^\infty \frac{\xi_\tau'(u)\hat Lu}{u_\tau} u_\tau^{\frac{\sigma(n-3)}{n-1}}\psi^2 - \frac{\sigma}{n-1}\int_{-\infty}^\infty \frac{\hat V_+(u_\tau-u\xi_\tau'(u))}{u_\tau} u_\tau^{\frac{\sigma(n-3)}{n-1}}\psi^2 \le  \int_{-\infty}^\infty (\psi_s)^2 u_\tau^{\frac{\sigma(n-3)}{n-1}}.
		\]
	Proceeding as in Theorem \ref{teo_main_critical_cs} we deduce that $\hat L$ is critical. The conclusion follows as in the first part of the proof.
\end{proof}

\begin{remark}
	If $\beta < \frac{4}{n-1}$ and the strict inequality $\Ric > -\beta V g$ in \eqref{eq_iporicci_2} holds at every point of $M$, the conclusion that $M$ must have only one end was obtained in \cite[Theorem 1.1]{xucheng}. 
\end{remark}

In view of Theorem \ref{teo_main_critical_2}, some of the applications of Cheeger \& Gromoll's theorem to the geometry of manifolds with $\Ric \ge 0$ directly extend to manifolds with spectral lower Ricci bounds. We mention, for instance, the following corollary that generalizes \cite{shen_sormani,car_ped}.
\begin{corollary}\label{cor_car_ped}
Let $(M,g)$ be a complete Riemannian manifold of dimension $n \ge 3$ satisfying the assumptions of Theorem \ref{teo_main_critical_2}. Then, one of the following cases occurs:
\begin{itemize}
	\item[(i)] $M$ is one-ended, $H^1_c(M) = 0$ and, if $M$ is orientable, $H_{n-1}(M,\mathbb{Z}) = 0$;
	\item[(ii)] $V \equiv 0$, $M$ is one-ended and it is the determinant line bundle of a compact manifold $P$ satisfying $\Ric_P \ge 0$, with locally the metric $\di t^2 + g_P$ where $t$ is the arclength of the fibers. 
	\item[(iii)] $V \equiv 0$ and $M = \R \times P$ with the product metric, for some compact orientable $P$ with $\Ric_P \ge 0$. 
\end{itemize}
\end{corollary}
\begin{remark}
	Note that $M$ deformation retracts onto $P$ in cases {\em (ii)} and {\em (iii)}, thus  
	\begin{itemize}
		\item[] in {\em (ii)}, $H_{n-1}(M,\mathbb{Z}) = 0$ if $M$ is orientable and $H_{n-1}(M,\mathbb{Z}) = \mathbb{Z}$ if $M$ is nonorientable;
		\item[] in {\em (iii)}, $H_{n-1}(M,\mathbb{Z}) = \mathbb{Z}$ if $M$ is orientable and $H_{n-1}(M,\mathbb{Z}) = 0$ if $M$ is nonorientable.
	\end{itemize}
\end{remark}

\begin{proof}
	The argument  \emph{verbatim} follows \cite[Proposition 5.3]{car_ped}, so we only sketch it. If $M$ has two ends, then {\emph (iii)} holds by Theorem \ref{teo_main_critical_2}. If $M$ has only on end, let $\pi : \hat M \to M$ be any twofold normal covering. The spectral condition in Theorem \ref{teo_main_critical_2} lifts to $\hat M$ with $\hat V = V \circ \pi$, so if $\hat M$ has two ends, then Theorem \ref{teo_main_critical_2} applies to give $\hat M = \R \times P$ and 
	\[
	M = \frac{\R \times P}{\langle {\rm Id}, \tau \rangle}
	\]
	for some isometry $\tau$ with $\tau^2 = {\rm Id}$. Since $P$ is compact and $\tau$ preserves lines, $\tau$ must be of the form $\tau(t,y) = (-t + a, f(y))$ for some isometry $f : P \to P$ and some $a \in \R$. Hence, case \emph{(ii)} occurs. Eventually, if both $M$ and $\hat M$ have only one end, case \emph{(i)} follows by \cite[Proposition 5.2]{car_ped}.
\end{proof}

In general, in case $(i)$ of Theorem \ref{teo_main_critical_2} the metric $\bar g = u^{2\beta}g$ may not be complete. To see this, consider Euclidean space, $V \equiv 0$ and $u$ a smoothing of the Green kernel centered at the origin:
\[
u(x) = \tau(|x|^{2-n})
\]
for a concave function $\tau : [0,\infty) \to \R$ satisfying $\tau(t) = t$ for $t \le 1$ and $\tau \equiv 2$ on $[4,\infty)$. Then, $\Delta u \le 0$ on $\R^n$ but $\bar g$ is incomplete for each
\[
\beta \in \left( \frac{1}{n-2}, \frac{4}{n-1} \right).
\]
The next theorem shows that for $\beta < (n-2)^{-1}$ the fact does not occur.

\begin{theorem}\label{teo_main_critical_3}
	Let $(M,g)$ be a complete Riemannian manifold of dimension $n \ge 3$, and assume that there exists $V \in C^{0,\alpha}_\loc(M)$ such that
	\begin{equation}\label{eq_iporicci_2}
		\Ric \ge - \beta V g \qquad \text{on } \, M, \qquad L \doteq -\Delta - V \ge 0
	\end{equation}
	with
	\[
	0<\beta < \frac{1}{n-2}.
	\]
	Then, for each $0 < u \in C^2(M)$ solving $Lu \ge 0$, the metric $\bar g = u^{2\beta} g$ is complete and satisfies
	\[
	\overline{\Ric}_f^N \ge \Ric + \beta V g \ge 0,
	\]
	where
	\[
	f = (n-2)\beta \log u \quad \text{and }\,\, \qquad N = n + \frac{\beta (n-2)^2}{1-\beta(n-2)}>n.
	\]
\end{theorem}

\begin{remark}\label{rem-border}
	In the borderline case $\beta = (n-2)^{-1}$, the same computations show that the $\infty$-Bakry-Emery Ricci tensor $\overline{\Ric}_f$ is non-negative on $M$ (cf. \cite[Proposition 3.3]{irv}). However, in this case we are not able to prove that $\bar g$ is complete, nor we are aware of counterexamples to this fact.
\end{remark}

\begin{proof}
	
	By \cite[Lemma 2.2]{mmrs}, one can construct a ``shortest ray" $\gamma$ in $(M,\bar g)$ issuing from a fixed origin $o$, that is, a divergent curve minimizing $\bar g$-distance between any pair of its points, with the property that $(M,\bar g)$ is complete if and only if $\gamma$ has infinite $\bar g$-length. Parametrizing $\gamma$ by $g$-arclength $s$, since $g$ is complete $\gamma$ is defined for $s \in [0,\infty)$; hence, to get completeness we only have to check that
	\[
	\ell_{\bar g}(\gamma_{[0, \infty)}) = \int_{0}^\infty u^\beta \di s = \infty.
	\]
	Since $\gamma$ is $\bar g$-minimizing, by equation \eqref{eq_fundamental} in Proposition \ref{prop_basiccomp} we deduce, using the hypothesis   $Lu\geq 0$ and $\Ric \ge - \beta V g $, that for every smooth function $\eta$ with compact support in $[0, \infty)$ and vanishing at $0$, and for every $\beta>0$,
	\begin{equation}\label{eq_DaProp2.9}
		\begin{array}{l}
			0 \le \disp (n-1)\int_{0}^{\infty} (\eta_{s})^2 {u^{-\beta}} - \beta \int_{0}^{\infty}\frac{(u_s)^2}{u^2} {\eta^2}{u^{-\beta}} - \beta(n-2)\int_{0}^{\infty} \frac{u_s}{u} \left({\eta^2}{u^{-\beta}}\right)_s.
		\end{array}
	\end{equation}
	Writing $\eta=u^\beta\psi$, we have
	\begin{align*}
		\eta^2u^{-\beta}&=u^{\beta}\psi^2,\\ \eta_s&=\beta\psi u^{\beta-1}u_s+u^\beta\psi_s,\\ (\eta_s)^2u^{-\beta}&=\beta^2\psi^2u^{\beta-2}(u_s)^2+u^\beta(\psi_s)^2+2\beta\psi\psi_su^{\beta-1}u_s,
	\end{align*}
	and substituting in equation \eqref{eq_DaProp2.9} we get
	\begin{align}\label{es1}
		(n-1)\int_0^{\infty} u^{\beta}(\psi_s)^2 \geq
		-2\beta\int_0^{\infty} \psi\psi_s u^{\beta-1}u_s +\beta(1-\beta)\int_0^{\infty}\psi^2u^{\beta-2}(u_s)^2.
	\end{align}
	Now we follow the same lines of the computations in \cite[Lemma 2.3]{cmr}. Define
	$$
	I \doteq \int_0^{\infty} \psi\psi_s u^{\beta-1}u_s
	$$
	and observe that
	$$
	I = \frac{1}{\beta} \int_0^{\infty} \psi\psi_s (u^{\beta})_s =-\frac{1}{\beta} \int_0^{\infty} u^\beta(\psi_s)^2 -\frac{1}{\beta} \int_0^{\infty} \psi\psi_{ss} u^{\beta}.
	$$
	Moreover, for every $t>1$ and using Young's inequality with $\eps>0$ to be chosen later, we have
	\begin{align*}
		2\beta I &=2\beta t I+2\beta(1-t)I\\ &=-2t \int_0^{\infty} u^\beta(\psi_s)^2 -2t \int_0^{\infty} \psi\psi_{ss} u^{\beta} +2\beta(1-t)\int_0^{\infty} \psi\psi_s u^{\beta-1}u_s\\
		&\leq -2t \int_0^{\infty} u^\beta(\psi_s)^2 -2t \int_0^{\infty} \psi\psi_{ss} u^{\beta}\\
		&\quad+\beta(t-1)\eps \int_0^{\infty}\psi^2u^{\beta-2}(u_s)^2 +\frac{\beta(t-1)}{\eps}\int_0^{\infty} u^\beta(\psi_s)^2.
	\end{align*}
	
	Recalling that $\beta < (n-2)^{-1}$ and choosing
	$$
	\eps \doteq \frac{1-\beta}{t-1}
	$$
	we obtain
	\begin{align*}
		2\beta I &\leq -2t \int_0^{\infty} \psi\psi_{ss} u^{\beta}+\beta(1-\beta)\int_0^{\infty}\psi^2u^{\beta-2}(u_s)^2\\
		&\quad+\left[\frac{\beta(t-1)^2}{1-\beta}-2t\right]\int_0^{\infty} u^\beta(\psi_s)^2.
	\end{align*}
	From \eqref{es1} we get
	\begin{equation}\label{eq_aux}
		0\leq \left[\frac{\beta(t-1)^2}{1-\beta}-2t+(n-1)\right]\int_0^{\infty} u^\beta(\psi_s)^2 -2t \int_0^{\infty} \psi\psi_{ss} u^{\beta}
	\end{equation}
	for every $t>1$ and every positive $\beta < (n-2)^{-1}$. Let
	$$
	P(t, n, \beta) \doteq \frac{\beta(t-1)^2}{1-\beta}-2t+(n-1) = \frac{1}{1-\beta}\sq{\beta t^2-2t + (n-1)-\beta(n-2)};
	$$
	it is not difficult to see that there exists $\bar{t} = \bar{t}(\beta, n)>1$ such that, for $n\geq 3$, $P(\bar{t}, n, \beta)<0$. Therefore, from \eqref{eq_aux} we deduce
	$$
	0\leq -\int_0^{\infty} u^{\beta}(\psi_s)^2 -C \int_0^{\infty} u^{\beta}\psi\psi_{ss}
	$$
	for every $\psi$ smooth with compact support in $[0,\infty)$ and vanishing at $0$, and for some positive constant $C$ depending on $\bar{t}$, $n$ and $\beta$. Now we choose $\psi=s\varphi$ with $\varphi$ smooth with compact support in $[0,\infty)$: thus
	$$
	\psi_s=\varphi+s\varphi_s,\quad\psi_{ss}=2\varphi_s+s\varphi_{ss},
	$$
	and we get
	$$
	\int_0^{\infty} u^{\beta}\varphi^2 \leq \int_0^{\infty} u^{\beta}\left(-2(C+1)s\varphi\varphi_s-Cs^2\varphi\varphi_{ss}-s^2(\varphi_s)^2\right).
	$$
	Choose $\varphi$ such that $\varphi\equiv 1$ on $[0,R]$, $\varphi\equiv 0$ on $[2R,\infty)$ and with $|\varphi_s|$ and $|\varphi_{ss}|$ bounded by $\tilde{C}/R$ and $\tilde{C}/R^2$, respectively, for $R\leq s\leq 2R$ ($\tilde{C}$ is a positive constant). Then
	$$
	\int_0^R u^{\beta} \leq \int_0^{\infty} u^{\beta}\varphi^2 \leq C \int_R^{\infty}u^{\beta}
	$$
	for some $C>0$ independent of $R$. We conclude that necessarily
	$$
	\int_0^{\infty}u^{\beta} = \infty,
	$$
	i.e. $\bar{g}=u^{2\beta} g$ is complete. The desired lower bound on the Bakry-Emery Ricci curvature follows by plugging $\varphi = \beta \log u$ into \eqref{eq_ric_weight}: in our assumptions, and using the inequality $|\di u|^2 g \ge \di u \otimes \di u$, we get
	\[
	\begin{array}{lcl}
	\overline{\Ric}_f & \ge & \disp - \beta V - (n-2)\beta^2 \di \log u \otimes \di \log u + \beta (V + |\di \log u|^2)g \\[0.2cm]
	& \ge & \disp \beta \big( 1 - (n-2)\beta\big) \di \log u \otimes \di \log u = \frac{1 - (n-2)\beta}{(n-2)^2 \beta} \di f \otimes \di f.
	\end{array}
	\]
	This concludes the proof.
\end{proof}

Two consequences of the previous theorem are the following topological properties. We first consider the compact case, exploiting the results in \cite{qian, lott, ww}.

\begin{corollary}\label{cor_car_1}
	Let $(M,g)$ be a compact Riemannian manifold of dimension $n \ge 3$, and assume that there exists $V \in C^{0,\alpha}_\loc(M)$ such that
	\begin{equation}\label{eq_iporicci_2}
		\Ric \ge - \beta V g \qquad \text{on } \, M, \qquad L \doteq -\Delta - V \ge 0,
	\end{equation}
	where
	\[
	0<\beta \leq \frac{1}{n-2}.
	\]
	Then:
	\begin{itemize}
		\item[$(i)$] $\pi_1(M)$ has a free abelian subgroup of finite index of rank $\leq n$, with equality iff $M$ is a flat torus and $V\equiv 0$.
		\item[$(ii)$] $\pi_1(M)$ is finite if $\Ric + \beta V g>0$ at one point.
	\end{itemize}
\end{corollary}
\begin{proof} Let $0 < u \in C^2(M)$ solve $Lu=0$. By Remark \ref{rem-border} the metric $\bar{g}=u^{2\beta}g$ satisfies
	\[
	\overline{\Ric}_f \ge \Ric + \beta V g \ge 0,
	\]
	with $f=(n-2)\beta \log u$. Moreover, since $M$ is compact $\bar g$ is complete. By \cite[Corollary 6.7]{ww}, $(ii)$ holds and $\pi_1(M)$ has a free abelian subgroup of finite index of rank $\leq n$. Moreover, equality holds if and only if $f$ is constant (i.e. $u$ is constant) and $(M,\bar{g})$ is a flat torus, hence so is $(M,g)$. Moreover, in this case, from $L u=0$ we also conclude $V\equiv 0$. 
\end{proof}

It is interesting to compare our theorem with those in \cite{boucar,car,carrose}, where the authors assume similar but different spectral conditions to obtain topological conclusions. In the non-compact case, using \cite{lim,XDli} we have the following:

\begin{corollary}\label{cor_car_2}
	Let $(M,g)$ be a complete, non-compact, Riemannian manifold of dimension $n \ge 3$, and assume that there exists $V \in C^{0,\alpha}_\loc(M)$ such that
	\begin{equation}\label{eq_iporicci_2}
		\Ric \ge - \beta V g \qquad \text{on } \, M, \qquad L \doteq -\Delta - V \ge 0,
	\end{equation}
	where
	\[
	0<\beta < \frac{1}{n-2}.
	\]
	Then:
	\begin{itemize}
		\item the conclusions of Corollary \ref{cor_car_ped} hold, and in case (i), $H_{n-1}(M,\mathbb{Z}) = 0$ holds independently of the orientability of $M$;
		\item positive harmonic functions on $M$ are constant.
	\end{itemize}
	
\end{corollary}

\begin{proof}
	For the first assertion, our range of $\beta$ is included in that of Corollary \ref{cor_car_ped}. The identity $H_{n-1}(M,\mathbb{Z}) = 0$ in case {\em (i)} regardless of orientability follows from \cite[Corollary 4.12]{lim} (see also its proof), an extension of \cite{shen_sormani} to manifolds with $\overline{\Ric}_f^N \ge 0$. The second property follows from \cite[Theorem 1.3]{XDli} applied to $(M, \bar g)$ once observed that, by \eqref{eq_Lapl_weighted}, a harmonic function $u$ satisfies $\bar{\Delta}_f u = 0$ on $(M,\bar g)$.
\end{proof}


\begin{remark}
	Compared to Corollary \ref{cor_car_ped}, the possibility to conclude $H_{n-1}(M,\mathbb{Z}) = 0$ in case {\em (i)} even for nonorientable $M$ depends on the following fact: the curvature bound 
	\[
	\overline{\Ric}_f^N \ge 0 
	\]
	granted by Theorem \ref{teo_main_critical_3} allow to split $M$ in presence of any $\bar g$-line, not only when $M$ has two ends. This is used to get information from the absence of the ``loop to infinity property" in \cite{shen_sormani}.
\end{remark}

%

\section{Applications to minimal hypersurfaces}\label{sec_geometry}

Let $M^n \to (N^{n+1},\bar g)$ be a complete, two-sided, minimally immersed hypersurface. Denote with a bar tensors associated to $\bar g$. Write
\[
J_\delta \doteq -\Delta - \delta\left( |A|^2 + \overline{\Ric}(\nu,\nu) \right)
\]
for the $\delta$-stability operator. Choose a Darboux frame $\{e_i\}$ along $M$, and set $e_0 = \nu$. As the second fundamental form $A$ in direction $\nu$ is traceless, it satisfies the refined Kato inequality
\[
\sum_{j=1}^n A_{1j}^2 \le \frac{n-1}{n}|A|^2,
\]
see \cite[Lemma 1.5]{prs_book}. Moreover, equality holds at a point $x$ where $A \neq 0$ if and only if $A(x)$ has only two eigenvalues $\lambda,\mu$, of multiplicities $1$ and $(n-1)$ respectively, and $Ae_1 = \lambda e_1$. By Gauss equation,
\begin{equation}\label{eq_gauss}
	\begin{array}{lcl}
		\Ric_{11} & = & \disp \sum_{\alpha \ge 2} \bar{R}_{1\alpha1\alpha} - \sum_{j \ge 1} A_{1j}^2 \ge \sum_{\alpha \ge 2} \bar{R}_{1\alpha1\alpha} - \frac{n-1}{n}|A|^2.
	\end{array}
\end{equation}

\begin{itemize}
	\item Assume that $\overline{\BRic} \ge 0$. Then,
	\begin{equation}\label{eq_gauss_1}
		\begin{array}{lcl}
			\Ric_{11} & \ge & \disp \sum_{\alpha \ge 2} \bar{R}_{1\alpha1\alpha} - \frac{n-1}{n}|A|^2 = \overline{\BRic}_{1\nu} - \overline{\Ric}(\nu,\nu) - \frac{n-1}{n}|A|^2 \\[0.5cm]
			& \ge & - \left(\overline{\Ric}(\nu,\nu) +|A|^2\right) + \frac{1}{n}|A|^2,
		\end{array}
	\end{equation}
	\item Assume that $\overline{\Sec}\ge 0$. Then,
	\begin{equation}\label{eq_gauss_2}
		\Ric_{11} \ge \disp \sum_{\alpha \ge 2} \bar{R}_{1\alpha1\alpha} - \frac{n-1}{n}|A|^2 \ge - \frac{n-1}{n} \left(\overline{\Ric}(\nu,\nu) +|A|^2\right).
	\end{equation}
\end{itemize}

We are ready to prove our results:

\begin{proof}[Proof of Theorem \ref{teo_minimal_main_stable}]
	Assume that $M$ has at least two ends, for otherwise $(i)$ holds. Defining $V \doteq \overline{\Ric}(\nu,\nu) +|A|^2$, from \eqref{eq_gauss_1} and the stability assumption we get
	\begin{equation}\label{eq_lower_ric_min}
	\Ric \ge - Vg + \frac{1}{n}|A|^2g , \qquad -\Delta - V \ge 0.
	\end{equation}
	Since $1 < \frac{4}{n-1}$ holds in dimension $n \in \{3,4\}$, applying Theorem \ref{teo_main_critical_2} we obtain that $V \equiv 0$ and that $M = \R \times P$ with the product metric $\di s^2 + g_P$, for some compact $(P,g_P)$ with non-negative Ricci curvature. From $0 = \Ric(\partial_s,\partial_s)$ and \eqref{eq_lower_ric_min} we conclude $|A| \equiv 0$, thus $\overline{\Ric}(\nu,\nu) = V - |A|^2 \equiv 0$. The conclusion  $\overline{\BRic}(\partial_s,\nu) \equiv 0$ follows from \eqref{eq_gauss_1}. In case $(i)$, if $M$ is parabolic then one cannot deduce $V \equiv 0$ (as the case of surfaces already suggests). However, if $\overline{\Ric} \ge 0$, any positive solution to $Lu = 0$ is superharmonic. Hence, the parabolicity of $M$ forces $u$ to be constant, so $V \equiv 0$.
\end{proof}	
	

If $\overline{\BRic}\ge 0$ is strengthened to $\overline{\Sec} \ge 0$, we gain one more dimension and further information in the previous result.

%
%

\begin{proof}[Proof of Theorem \ref{teo_minimal_main_stable_sec}]
	The proof follows the same path as that of Theorem \ref{teo_minimal_main_stable}, we only point out the differences: setting $V \doteq \overline{\Ric}(\nu,\nu) +|A|^2$, Gauss equation \eqref{eq_gauss_2} yields
	\begin{equation}\label{eq_forsec}
		\Ric \ge - \frac{n-1}{n} V, \qquad -\Delta - V \ge 0.
	\end{equation}
	Inequality $\frac{n-1}{n} < \frac{4}{n-1}$ now holds up to $n=5$, justifying the dimensional improvement. If $M$ has more than one end, from $V \equiv 0$ one deduces $|A| \equiv 0$ and $\overline{\Ric}(\nu,\nu) \equiv 0$. By Gauss equation, $M$ has non-negative sectional curvature, hence so does $(P,g_P)$.\\
\end{proof}	
	
\begin{proof}[Proof of Corollary \ref{cor_topol_minimal}]
It is a direct application of Corollaries \ref{cor_car_ped} and \ref{cor_car_2}, once we recall \eqref{eq_forsec} in the assumptions of Theorem \ref{teo_minimal_main_stable_sec} and that $\frac{n-1}{n} < \frac{1}{n-2}$ holds in dimension $3$.
\end{proof}

We are left to consider the case of $\frac{n-2}{n}$-stable hypersurfaces. Cheng \& Zhou's strategy to obtain Theorem \ref{teo_chengzhou} is based on a clever refinement of \cite[Theorem A]{liwang_ens}, and exploits the Hardy inequality
\[
\int_M \frac{(n-2)^2}{4r^2} \phi^2 \le \int_M |\nabla \phi|^2 \qquad \forall \, \phi \in \lip_c(M)
\]
satisfied by minimal hypersurfaces in Euclidean space, see Remark \ref{rem_hardy}.
Crucial for the conclusion in \cite{cheng_zhou_1} are both the completeness of the metric $r^{-2} g$ on $M$, and a lower bound for the conformal factor $r^{-2}$ in terms of the intrinsic distance on $M$. When trying to adapt the ideas to more general ambient spaces, this latter control is hard to achieve. In dimension $n=3$, we can use Theorem \ref{teo_main_critical} to overcome the problem and obtain information on $1/3$-stable hypersurfaces. The next theorem generalizes Theorem \ref{cor_chengzhou} in the Introduction.

\begin{theorem}\label{teo_minimal_main}
	Let $M^3 \to (N^{4},\bar g)$ be a complete, non-compact, two-sided minimal hypersurface in a manifold satisfying $\overline{\Sec}\ge 0$. Assume that $M$ is $\delta$-stable with $\delta \ge 1/3$. Then, one of the following mutually exclusive cases occurs:
	\begin{itemize}
		\item[(i)] $M$ has only one end. Moreover, either $M$ is non-parabolic, or $M$ is parabolic, totally geodesic and $\overline{\Ric}(\nu,\nu) \equiv 0$;
		\item[(ii)] $M$ is totally geodesic, $\overline{\Ric}(\nu,\nu) \equiv 0$ and $M =  \R \times P$ with the product metric, for some compact surface $P$ with non-negative Gaussian curvature.
		\item[(iii)] $\delta = 1/3$, $\overline{\Ric}(\nu,\nu) \equiv 0$, $M$ has at least two ends and is non-parabolic, the set $U \doteq \{|A|>0\}$ is non-empty and $A$ has only two eigenvalues $\lambda, \mu$ on each connected component of $U$, of multiplicities $1$ and $2$ respectively. Moreover, if $\{e_j\}$ is an orthonormal frame of eigenvectors with $Ae_1 = \lambda e_1$ and $Ae_\alpha = \mu e_\alpha$ for $\alpha \ge 2$, and setting $e_0 = \nu$, the components of the curvature tensor of $N$ satisfy
		\begin{equation}\label{eq_condiR}
			\begin{array}{lcl}
				\bar{R}_{1010} = \bar R_{0\beta 0 \alpha} = \bar R_{1\beta 1 \alpha} = \bar R_{\beta \alpha \beta 0} = \bar R_{\beta \alpha \beta 1} = \bar{R}_{\alpha 0 \alpha 1} = \bar{R}_{101\alpha} = \bar{R}_{010\alpha}  = 0.
			\end{array}
		\end{equation}
	\end{itemize}	
\end{theorem}

\begin{proof}
	The restriction to dimension $n=3$ is only needed at one step of the proof, hence we prefer to keep writing $n$. The assumed $\delta$-stability and $\overline{\Sec} \ge 0$ therefore imply
	\[
	J_{\frac{n-2}{n}} \doteq -\Delta - \frac{n-2}{n} \left( |A|^2 + \overline{\Ric}(\nu,\nu)\right) \ge 0.
	\]
	Assume that $M$ has at least two ends, for otherwise $(i)$ holds. Define
	\[
	L \doteq -\Delta - \frac{n-2}{n}|A|^2, \qquad V \doteq \frac{n-2}{n}|A|^2.
	\]
	By Gauss equation \eqref{eq_gauss}, 	
	\[
	\Ric_{11} = \sum_{\alpha \ge 2} \bar{R}_{1\alpha1\alpha} - \sum_{j \ge 1} A_{1j}^2 \ge - \frac{n-1}{n}|A|^2 = - \frac{n-1}{n-2}V
	\]
	On the other hand,
	\begin{equation}\label{eq_compastabi}
		L \ge J_{\frac{n-2}{n}} \ge 0.
	\end{equation}	
	Therefore, applying Theorem \ref{teo_main_critical} we deduce that $L$ is critical and that, for each $x \in M$, there exists an orthonormal basis $\{e_j\}$ whose first element $e_1$ satisfies
	\begin{equation}\label{eq_ref_kat}
	\Ric_{11} = \sum_{j \ge 1} A_{1j}^2 = - \frac{n-1}{n}|A|^2.
	\end{equation}
	Namely, $\bar{R}_{1\alpha 1\alpha} = 0$ for each $\alpha \ge 2$, and $A$ satisfies equality in the refined Kato inequality. From \eqref{eq_compastabi} and the criticality of $L$,
	\[
	\overline{\Ric}(\nu,\nu) \equiv 0 \qquad \text{on } \, M.
	\]	
	If $\delta > \frac{n-2}{n}$, we deduce from
	\[
	L \ge J_\delta \ge J_{\frac{n-2}{n}} \ge 0
	\]
	and the criticality of $L$ that $J_\delta = J_{\frac{n-2}{n}}$, which means $|A| \equiv 0$. From Gauss equation we deduce that $M$ has non-negative sectional curvature. Since $M$ has two ends, the splitting theorem implies $M = \R \times P$ for some compact manifold $P$ with non-negative sectional curvature.\\[0.2cm]
	It remains to consider the case $\delta = \frac{n-2}{n}$ and $U \doteq \{|A|>0\} \neq \emptyset$. By \eqref{eq_ref_kat} and the characterization of equality in the refined Kato inequality,
	\begin{equation}\label{eq_frame}
		A_{1\alpha} = 0, \quad A_{11} = \lambda, \quad A_{\alpha\beta} = \mu \delta_{\alpha\beta} \qquad \text{with } \, \lambda + (n-1)\mu = 0.
	\end{equation}
	Since $\lambda \neq \mu$, the two eigenspace distributions have constant multiplicity and are therefore smooth. Thus, $\lambda, \mu \in C^\infty(U)$ and the frame $\{e_1,e_\alpha\}$ satisfying \eqref{eq_frame} can be chosen smoothly around any given point. Furthermore, from $\bar R_{1\alpha1\alpha}=0$ and $\overline{\Ric}(\nu,\nu) \equiv 0$ we deduce
	\[
	0 = \bar R_{1\alpha 1 \alpha} = \bar R_{1010} = \bar R_{\alpha 0 \alpha 0} \qquad \forall \alpha \ge 2.
	\]
	For $i \in \{0,1\}$, set
	\[
	F_i(t) = \bar R(e_i, e_\alpha + te_\beta, e_i, e_\alpha + te_\beta).
	\]
	From $F_i \ge 0$ and $F_i(0)=0$ we deduce $0 = F_i'(0) = 2\bar R_{i\beta i \alpha}$. Similarly, setting
	\[
	F_i(t) = \bar R(e_\beta, e_i + te_\alpha,e_\beta, e_i + te_\alpha),
	\]
	from $F_i \ge 0$ and $F_i(0)=0$ we deduce $0 = F_i'(0) = 2 \bar R_{\beta \alpha \beta i}$. Likewise, differentiating
	\[
	F(t) = \bar R( e_\alpha, e_0 + te_1,e_\alpha, e_0 + te_1)
	\]
	we get $0 = \bar R_{\alpha 0 \alpha 1}$, and differentiating, respectively,
	\[
	F(t) = \bar R( e_1, e_0 + te_\alpha, e_1, e_0 + te_\alpha), \qquad F(t) = \bar R( e_0, e_1 + te_\alpha, e_0, e_1 + te_\alpha)
	\]
	we deduce, respectively, $0 = \bar R_{101\alpha}$ and $0 = \bar R_{010\alpha}$. This concludes the proof.	
\end{proof}

	\begin{remark}
		Theorem \ref{teo_minimal_main} holds in dimension $n\ge 4$, with $(n-2)/n$ replacing $1/3$ and principal eigenspaces of dimensions $1$ and $(n-1)$,  whenever $M$ supports a positive solution to $J_{\frac{n-2}{n}} w \ge 0$ satisfying
		\[
		w(x) = o \left( r(x)^{\frac{n-2}{n-3}} \right) \qquad \text{as } \, r(x) \to \infty.
		\]
		However, currently we do not have a manageable geometric condition to guarantee the existence of such $w$.
	\end{remark}
	
	In $(iii)$ of Theorem \ref{teo_minimal_main}, it is unclear to us whether the mixed components $\bar R_{01\alpha\beta}$ should vanish or not. If this were the case, a computation using the Codazzi \& Mainardi equations would imply that the distribution generated by $\{e_\alpha\}$ is integrable, allowing to locally (or even globally) split $M$ as a warped product. This is the case of ambient Euclidean space:
	
	\begin{proof}[Proof of Theorem \ref{cor_chengzhou}]
		Denote by $f : M \to \R^4$ the minimal immersion. If $M$ has at least two ends, by Theorem \ref{teo_minimal_main} and since case $(ii)$ does not occur in Euclidean space we deduce that $M$ is non-parabolic, that $U \doteq \{|A|>0\}$ is non-empty and that $A$ has only two eigenvalues on any connected component $U' \subset U$. By a result due to Do Carmo \& Dajczer \cite[Theorem 4.4]{DoCD}, $f(U')$ is a piece of a catenoid $C$, and by continuity so is $f(\overline{U'})$. If by contradiction $U' \not \equiv M$, then $A$ should vanish on $\partial U'$, hence the second fundamental form of $C$ should vanish on $f(\partial U')$. However, the second fundamental form of $C$ is positive everywhere, contradiction. Therefore, $U'\equiv M$ and $f(M) \subset C$. Since $M$ is complete and $f : M \to C$ is a local isometry, by Ambrose theorem $f$ is surjective and a Riemannian covering, hence a diffeomorphism since the $3$-dimensional catenoid is simply connected.
	\end{proof}
	
	\begin{proof}[Proof of Theorem \ref{teo_bernstein}]
		Since $\R^4$ is orientable and $M$ is two-sided, then $M$ is orientable. Moreover, Theorem \ref{cor_chengzhou} guarantees that $M$ has only one end: indeed, the catenoid is not $\delta$-stable for any $\delta > 1/3$, as inferrable by the criticality of $\Delta + 1/3|A|^2$ in Example \ref{ex_catenoid}.
In our assumptions, setting $V = \delta |A|^2$ we have from \eqref{eq_gauss}
		\[
		\Ric \ge - \frac{2}{3}|A|^2 g = -\beta Vg, \qquad -\Delta - V \ge 0 \qquad \text{with } \, \beta = \frac{2}{3\delta} < 2.
		\]
		We can thus apply Corollary \ref{cor_car_ped} to deduce either $H_2(M, \mathbb{Z}) = 0$ or that $M$ is, in particular, a manifold with non-negative Ricci curvature and linear volume growth. The second case does not occur, as Gauss equation would imply $0 \le R = -|A|^2$, thus $M$ would be totally geodesic, contradicting the linear volume growth property. We claim that $H_2(M,\mathbb{Z})=0$ has the following topological implication: \\[0.2cm]
		\noindent \textbf{Claim 1:} for any smooth, connected, relatively compact open set $U \Subset M$, each connected component $L$ of $M \backslash U$ has connected boundary.\\[0.2cm]
		{\em Proof:} assume by contradiction that $\partial_1 L$ and $\partial_2 L$ are distinct components of $\partial L \subset \partial U$. Pick tubular neighbourhoods $T_j$ of $\partial_j L$ and points $x_j,y_j \in T_j$ with $x_j \in U$, $y_j \in L$. Joining $y_1$ to $y_2$ with an arc in $L$, $x_1$ to $x_2$ with an arc in $U$ and $x_j$ to $y_j$ with a suitable arc in $T_j$ transverse to $\partial_j L$, we produce a loop $\gamma$ with intersection number $\gamma \cdot \partial_j L = \pm 1$. This contradicts the fact that $[\partial_j L] = 0$ in $H_2(M,\mathbb{Z})$. We include a brief argument to see this. Let $\Sigma = \partial_1L$,  pick a chain $c$ so that $\Sigma = \partial c$ and a smooth, connected relatively compact open set $V$ containing $\Sigma$ and the support of $c$. By \cite[Theorem 18.7]{lee}, $c$ can be chosen smooth. Let $P$ be the oriented double of $\overline{V}$. Then, $[\Sigma] = 0$ in $H_2(P,\mathbb{Z})$, so the following functional is zero by Stokes theorem for chains in \cite[Theorem 18.12]{lee}:
		\[
		[\omega] \mapsto \int_{\Sigma} \omega \ \ : \ \ H^2(P) \to \R.
		\]
		By Poincar\'e duality, the Poincar\'e dual $\eta_\Sigma$ is therefore zero in cohomology. Hence, by \cite[Proposition 7.3.12]{nicolaescu}, the intersection number satisfies
		\[
		\pm 1 = \gamma \cdot \Sigma = \int_P \eta_\gamma \wedge \eta_\Sigma = 0,
		\]
		contradiction.\\[0.2cm]
		Having shown that $M$ has only one end and that $H_2(M,\mathbb{Z}) = 0$, to prove that $M$ is a hyperplane we follow \cite{cl_2,clms,hlw} with some adjustments\footnote{Indeed, it is not enough to apply \cite[Theorem 1.7]{hlw} because $M$ is assumed to be simply connected there, and properness does not lift to the universal covering unless $M$ has finite fundamental group.}. Let $r$ be the distance in $\R^4$ from $0$ and write $\mathbb{B}(R)$ for the ball of radius $R$ in $\R^4$ centered at $0$. Up to translation, we assume $0 \not \in M$, $M \cap \mathbb{B}(1) \neq \emptyset$. Fix a connected component $M_1^*$ of $M \cap \mathbb{B}(1)$, and for $R>1$ let $M_R^*$ be the connected component of $M \cap \mathbb{B}(R)$ containing $M_1^*$. We shall prove that
		\begin{equation}\label{eq_cubic}
		|M^*_R|_g \le \Lambda R^3 \qquad \forall \, R>1,
		\end{equation}
		for some $\Lambda \in \R^+$. Write for convenience $c = e^{8\pi \sqrt{2}}$. Let $M_{cR}$ be the union of $M^*_{cR}$ and the relatively compact connected components of its complement. Since $M$ has only one end, $M \backslash M_{cR}$ is connected and thus, by Claim 1, $\partial M_{cR} = \partial (M\backslash M_{cR})$ is connected.
		Denote by $N$ the manifold $M$ endowed with the complete metric $\tilde g = r^{-2}g$, and by $N_0$ the subset $M_{cR}$. By \cite[Lemma 25]{cl_2}, for any $p \in M \backslash M_{cR}$ and $q \in M^*_R$ it holds
		\begin{equation}\label{eq_triangle_1}
		\di_{\tilde g}(p,q) \ge \log \frac{r(p)}{r(q)} \ge 8 \pi \sqrt{2}.
		\end{equation}
		We can thus apply \cite[Theorem 5.4]{hlw} to deduce the existence of a smooth, connected, relatively compact open set $\Omega \subset N_0$ with
		\begin{itemize}
			\item[(i)] $\partial N_0 \subset \Omega \subset B^{\tilde g}_{4\pi\sqrt{2}}(\partial N_0)$, the tubular neighbourhood of $\partial N_0$ of $\tilde g$-radius $4\pi\sqrt{2}$;
			\item[(ii)] every connected component of $\partial'\Omega \doteq \partial \Omega \backslash \partial N_0$ is a sphere of $\tilde g$-area at most $16\pi$ and diameter at most $4\pi$.
		\end{itemize}
		By (i), \eqref{eq_triangle_1} and the triangle inequality, $M_R^*$ is disjoint from $(M \backslash M_{cR}) \cup \overline{\Omega}$. Moreover, applying \cite[Lemma 25]{cl_2} with $p \in \overline \Omega$ and $q \in \partial M_{cR}$ we have $r(p) \le R e^{12\pi \sqrt{2}}$. In particular,
		\begin{itemize}
			\item[(iii)] $\tilde g \ge C_1R^{-2} g$ on $\partial' \Omega$, for some constant $C_1$.
		\end{itemize}
		Let $L$ be the union of $(M \backslash M_{cR}) \cup \Omega$ and the connected components of its complement which are disjoint from $M_R^*$, and let $M' = M \backslash L$. Then, $M'$ is relatively compact and $\partial M' \subset \partial'\Omega$. Moreover, since $M_R^*$ is connected, then $M' \supset M^*_R$ and $M'$ is connected as well. By Claim 1, $\partial M' = \partial L$ is connected, hence by (ii) its $\tilde g$-area does not exceed $16\pi$. The isoperimetric inequality for minimal hypersurfaces and (iii) then imply
		\[
		|M_R^*|_g \le |M'|_g \le C_2 |\partial M'|_g^{\frac{3}{2}} \le C_3 R^3|\partial M'|_{\tilde g}^{\frac{3}{2}} \le C_4 R^3,
		\]
		as required. To conclude, observe that $\{M_R^*\}$ is a family of relatively compact domains exhausting $M$. Applying the curvature estimate in \cite[Corollary 1.5]{hlw} to each (scaled) $M^*_R$ and letting $R \to \infty$ we deduce that $|A| \equiv 0$.
		%
		%
	\end{proof}

\section*{Acknowledgements}

We thank Gioacchino Antonelli, William Minicozzi, Marco Pozzetta, Giona Veronelli, Wilson Cunha and Marco Radeschi for useful remarks, and Gaoming Wang for clarifying for us a point in \cite{hlw}. We also thank the referees for their careful reading, and especially one of them for suggesting a simpler argument for Theorem \ref{teo_subcrit} based on \cite{pincho_psa}, which we reproduced in the revised version. The first, third and fourth authors are members of GNSAGA (Gruppo Nazionale per le Strutture Algebriche, Geometriche e loro Applicazioni). All the authors are partially supported by the PRIN project no. 20225J97H5 (Italy) ``Differential-geometric aspects of manifolds via Global Analysis''.

\vspace{0.4cm}

\noindent \textbf{Conflict of Interest.} The authors have no conflict of interest.



\

\Addresses

\end{document}